\documentclass[12pt, a4paper]{article}

\title{The kernel of chromatic quasisymmetric functions on graphs and hypergraphic polytopes}
\author{Raul Penaguiao\footnote{e-mail:raul.penaguiao@math.uzh.ch}\blfootnote{Institute of Mathematics, University of Zurich, Winterthurerstrasse 190, Zurich, CH - 8057.}}
\date{March 27, 2020}

\usepackage{graphicx}
\usepackage{amsmath}
\usepackage{amssymb}
\usepackage{amsthm}
\usepackage[utf8]{inputenc}
\usepackage[english]{babel}
\usepackage{bbold}
\usepackage{tikz-cd}
\usepackage{xr}
\usetikzlibrary{babel}
\usepackage{todonotes}
\usepackage{fullpage}
\usepackage{hyperref}
\usepackage{bm}
\usepackage{wrapfig}
\usepackage[capitalise]{cleveref}

\newtheorem{thm}{Theorem}
\newtheorem{prop}[thm]{Proposition}
\newtheorem{lm}[thm]{Lemma}
\newtheorem{cor}[thm]{Corollary}
\newtheorem{conj}[thm]{Conjecture}

\theoremstyle{remark}
\newtheorem{defin}[thm]{Definition}

\newtheorem{rem}[thm]{Remark}
\newtheorem{smpl}[thm]{Example}
\crefname{lm}{Lemma}{Lemmas}
\crefname{thm}{Theorem}{Theorems}
\crefname{prop}{Proposition}{Propositions}

\newcommand{\hs}[1]{\overline{\mathbf{#1}}}
\newcommand{\lhs}{\overline{E}}
\newcommand{\oPi}{\mathbf{C}}
\newcommand{\opi}{\vec{\boldsymbol{\pi}}}
\newcommand{\otau}{\vec{\boldsymbol{\tau}}}
\newcommand{\olambda}{\vec{\boldsymbol{\gamma}}}
\newcommand{\ogamma}{\vec{\boldsymbol{\gamma}}}
\newcommand{\odelta}{\vec{\boldsymbol{\delta}}}
\newcommand{\gpHa}{\mathbf{GP}}
\newcommand{\gHa}{\mathbf{G}}
\newcommand{\citestan}{\cite[Proposition 3.2]{gebhard99}}
\newcommand{\citeaguHM}{\cite[Definition 12.19]{aguiar10}}
\newcommand{\citeaguWQ}{\cite[Section 17.3.1]{aguiar10}}
\newcommand{\parpi}{\boldsymbol{\pi}}
\newcommand{\partau}{\boldsymbol{\tau}}
\newcommand{\makepar}{\boldsymbol{\lambda}}

\newcommand{\uhsm}{\boldsymbol{\Upsilon }}
\newcommand{\bfa}{\mathbf{a}}
\newcommand{\fff}{\mathcal{K}}
\newcommand{\bff}{\overline{\mathcal{K}}}
\newcommand\blfootnote[1]{%
  \begingroup
  \renewcommand\thefootnote{}\footnote{#1}%
  \addtocounter{footnote}{-1}%
  \endgroup
}

\DeclareMathOperator{\im}{im}
\DeclareMathOperator{\comu}{comu}
\DeclareMathOperator{\Orth}{Orth}
\DeclareMathOperator{\pt}{pt}
\DeclareMathOperator{\Res}{Res}

\DeclareMathOperator{\twist}{twist}
\DeclareMathOperator{\conv}{conv}
\DeclareMathOperator{\spn}{span}
\DeclareMathOperator{\id}{id}
\DeclareMathOperator{\inc}{inc}

\begin{document}

\maketitle

\abstract{
The chromatic symmetric function on graphs is a celebrated graph invariant.
Analogous chromatic maps can be defined on other objects, as presented by Aguiar, Bergeron and Sottile.
The problem of identifying the kernel of some of these maps was addressed by F\'eray, for the Gessel quasisymmetric function on posets.

On graphs, we show that the modular relations and isomorphism relations span the kernel of the chromatic symmetric function.
This helps us to construct a new invariant on graphs, which may be helpful in the context of the tree conjecture.
We also address the kernel problem in the Hopf algebra of generalized permutahedra, introduced by Aguiar and Ardila.
We present a solution to the kernel problem on the Hopf algebra spanned by hypergraphic polytopes, which is a subfamily of generalized permutahedra that contains a number of polytope families.

Finally, we consider the non-commutative analogues of these quasisymmetric invariants, and establish that the word quasisymmetric functions, also called non-commutative quasisymmetric functions, form the terminal object in the category of combinatorial Hopf monoids.
As a corollary, we show that there is no combinatorial Hopf monoid morphism between the combinatorial Hopf monoid of posets and that of hypergraphic polytopes. 
}


\section{Introduction}

\subsection*{Chromatic function on graphs}
For a graph $G$ with vertex set $V(G)$, a coloring $f$ of the graph $G$ is a function $f:V(G) \to \mathbb{N} $. A coloring is \textit{proper} in $G$ if no edge is monochromatic.

We denote by $\mathbf{G} $ the graph Hopf algebra, which is a vector space freely generated by the graphs whose vertex sets are of the form $[n ] $ for some $n \geq 0$.
This can be endowed with a Hopf algebra structure, as described by Schmitt in \cite[Chapter 12]{schmitt94}, and also presented below in \cref{sec:halg}.

Stanley defines in \cite{stanley95} the \textit{chromatic symmetric function} of $G$ in commuting variables $ \{x_i\}_{i\geq 1}$ as
\begin{equation}\label{eq:csfdefin}
\Psi_{\gHa}(G) = \sum\nolimits_f x_f \, , 
\end{equation}
where we write $x_f = \prod_{v \in V(G)} x_{f(v)} $, and the sum runs over proper colorings of $G$.
Note that $\Psi_{\mathbf{G}}(G) $ is in the ring \textit{Sym} of symmetric functions. 
The ring $Sym$ is a Hopf subalgebra of $QSym$, the ring of quasisymmetric functions introduced by Gessel in \cite{gessel84}.
A long standing conjecture in this subject, commonly referred to as the \textit{tree conjecture}, is that if two trees $T_1, T_2$ are not isomorphic, then $\Psi_{\mathbf{G}}(T_1) \neq \Psi_{\mathbf{G}}(T_2)$.

When $V(G) = [n]$, the natural ordering on the vertices allows us to consider a non-commutative analogue of $
\Psi_{\mathbf{G}}$, as done by Gebhard and Sagan in \cite{gebhard99}.
They define the chromatic symmetric function on non-commutative variables $ \{\mathbf{a}_i\}_{i\geq 1}$ as
$$ \boldsymbol{\Upsilon }_{\gHa} (G) = \sum\nolimits_f \mathbf{a}_f \, , $$
where we write $\mathbf{a}_f = \bfa_{f(1)} \dots \bfa_{f(n)} $, and we sum over the proper colorings $f$ of $G$.

Note that $ \boldsymbol{\Upsilon }_{\gHa} (G) $ is homogeneous and symmetric in the variables $\{ \mathbf{a}_i\}_{i\geq 1}$.
Such power series are called \textit{word symmetric functions}.
The ring of word symmetric functions, \textbf{WSym} for short, was introduced in \cite{rosas06}, and is sometimes called the ring of symmetric functions in non-commutative variables, or \textbf{NCSym}, for instance in \cite{bergeron09}.
Here we adopt the former name to avoid confusion with the ring of non-commutative symmetric functions.

In this paper we describe generators for $\ker \Psi_{\gHa} $ and $\ker \boldsymbol{\Upsilon }_{\gHa}$.
A similar problem was already considered for posets.
In \cite{feray15}, F\'eray studies $\Psi_{\mathbf{Pos}}$, the Gessel quasisymmetric function defined on the poset Hopf algebra, and describes a set of generators of its kernel.

Some elements of the kernel of $\Psi_{\gHa}$ have already been constructed in \cite{paquet13} by Guay-Paquet and independently in \cite{orellana14} by Orellana and Scott.
These relations, called \textit{modular relations}, extend naturally to the non-commutative case.
We introduce them now.

Given a graph $G$ and an edge set $E $ that is disjoint from $E(G)$, let $G\cup E$ denote the graph $G$ with the edges in $E$ added.
If we have edges $e_3\in G$ and $e_1, e_2 \not \in G$ such that $ \{ e_1, e_2, e_3 \}$ forms a triangle, then we also have
\begin{equation}\label{eq:modrelgraphs}
\boldsymbol{\Upsilon }_{\gHa}(G) - \boldsymbol{\Upsilon }_{\gHa}(G \cup \{e_1\} ) - \boldsymbol{\Upsilon }_{\gHa}(G \cup \{e_2\} ) + \boldsymbol{\Upsilon }_{\gHa}(G \cup \{e_1, e_2\} ) = 0 \, .
\end{equation}

\begin{figure}[h]
	\centering
	\includegraphics[width=238pt]{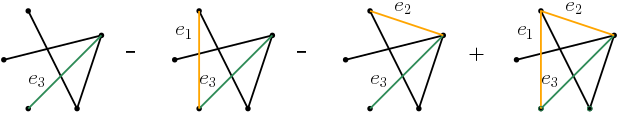}
	\caption{\label{fig:modrel}Example of a modular relation.}
\end{figure}
We call the formal sum $G - G \cup \{e_1\} - G \cup \{e_2\} + G \cup \{e_1, e_2\}$ in $\mathbf{G}$ a \textit{modular relation on graphs}.
An example is given in \cref{fig:modrel}.
Our first result is that these modular relations span the kernel of the chromatic symmetric function in non-commuting variables.
The structure of the proof also allows us to compute the image of the map.

\begin{thm}[Kernel and image of $\boldsymbol{\Upsilon }_{\gHa} : \mathbf{G } \to \mathbf{WSym} $]\label{thm:graphkernel}
The modular relations span $\ker \boldsymbol{\Upsilon }_{\gHa}$.
The image of $\boldsymbol{\Upsilon }_{\gHa} $ is \textbf{WSym}.
\end{thm}

Two graphs $G_1, G_2$ are said to be isomorphic if there is a bijection between the vertices that preserves edges.
For the commutative version of the chromatic symmetric function, if two isomorphic graphs $G_1, G_2$ are given, it holds that $\Psi_{\gHa}(G_1)$ and $\Psi_{\gHa}(G_2)$ are the same.	
The formal sum in $\mathbf{G}$ given by $G_1 - G_2$ is called an \textit{isomorphism relation on graphs}.

\begin{thm}[Kernel and image of $\Psi_{\gHa} : \mathbf{G } \to Sym $]\label{thm:graphcomukernel}
The modular relations and the isomorphism relations generate the kernel of the commutative chromatic symmetric function $\Psi_{\gHa} $.
The image of $\Psi_{\gHa} $ is \textit{Sym}.
\end{thm}

It was already noticed that $\Psi_{\gHa}$ is surjective.
For instance, in \cite{willigenburg15}, several bases of $Sym_n $ are constructed, which are the chromatic symmetric function of graphs, namely are of the form $\{\Psi_{\gHa}( G_{\lambda}) | \lambda \vdash n \}$ for suitable graphs $G_{\lambda }$ on $n$ vertices.
Here we present a new such family of graphs.

At the end of \cref{sec:3} we introduce a new graph invariant $\tilde{\Psi}$, called the \textit{augmented chromatic invariant}.
We observe that modular relations on graphs are in the kernel of the augmented chromatic invariant.
It follows from \cref{thm:graphcomukernel} that $\ker \Psi_G = \ker \tilde{\Psi} $.
This reduces the tree conjecture in $\Psi_{\gHa}$ to a similar conjecture on this new invariant $\tilde{\Psi}$, which contains seemingly more information.
\medskip


\subsection*{Generalized Permutahedra}

Another goal of this paper is to look at other kernel problems of chromatic flavor.
In particular, we establish similar results to \cref{thm:graphkernel,thm:graphcomukernel} in the combinatorial Hopf algebra of hypergraphic polytopes, which is a Hopf subalgebra of generalized permutahedra.

Generalized permutahedra form a family of polytopes that include permutahedra, associahedra and graph zonotopes.
This family has been studied, for instance, in \cite{postnikov08}, and we introduce it now.

The Minkowski sum of two polytopes $\mathfrak{a}, \mathfrak{b}$ is set as $\mathfrak{a}+ \mathfrak{b} = \{a + b |\,  a \in \mathfrak{a}, \, b\in \mathfrak{b} \}$.
The Minkowski difference $\mathfrak{a}- \mathfrak{b}$ is only sometimes defined: it is the unique polytope $\mathfrak{c}$ that satisfies $\mathfrak{b} + \mathfrak{c} = \mathfrak{a}$, if it exists.
We denote as $\sum_i \mathfrak{a}_i$ the Minkowski sum of several polytopes.

If we let $\{e_i | i \in I \}$ be the canonical basis of $\mathbb{R}^I$, a \textit{simplex} is a polytope of the form $\mathfrak{s}_J  = conv\{ e_j | j \in J \} $ for non-empty $J \subseteq  I$.
A generalized permutahedron in $\mathbb{R}^I$ is a polytope given by real numbers $\{a_J \}_{\emptyset \neq J \subseteq I}$ as follows:
Let $A_+ = \{J |   a_J>0 \} $ and $ A_- = \{J |   a_J<0 \}  $.
Then, the corresponding generalized permutahedron is
\begin{equation}\label{eq:gpcoefs}
\mathfrak{q} = \left( \sum_{J \in A_+ } a_J \mathfrak{s}_J \right) - \left( \sum_{J\in A_-} |a_J| \mathfrak{s}_J \right) \, ,
\end{equation}
if the Minkowski difference exists.
We identify a generalized permutahedron $\mathfrak{q}$ with the list $\{a_J \}_{\emptyset \neq J \subseteq I}$.
Note that not every list of real numbers will give us a generalized permutahedron, since the Minkowski difference is not always defined.

In \cite{postnikov09}, generalized permutahedra are introduced in a different manner.
A polytope is said to be a generalized permutahedron if it can be described as 
$$ \mathfrak{q} = \left\{ x \in \mathbb{R}^n | \sum_{ i \in I} x_i \leq z_I \text{ for } I \subsetneq [n] \text{ non-empty }; \sum_{i\in [n] } x_i = z_{[n]} \right\}  \, ,$$
for reals $\{z_J \}_{\emptyset \neq J \subseteq I}$.

A third definition of generalized permutahedra is present in \cite{aguiar17}.
Here, a generalized permutahedron is a polytope whose normal fan coarsens the one of the permutahedron.
These three definitions are equivalent, and a discussion regarding this can be seen in \cref{sec:2gp}.

A \textit{hypergraphic polytope} is a generalized permutahedron where the coefficients $a_J$ in \eqref{eq:gpcoefs} are non-negative.
For a hypergraphic polytope $\mathfrak{q}$, we denote by $\mathcal{F}(\mathfrak{q}) \subseteq 2^I\setminus  \{ \emptyset \} $ the family of sets $J\subseteq I$ such that $a_J > 0$.
A \textit{fundamental hypergraphic polytope} on $\mathbb{R}^I$ is a hypergraphic polytope $\sum_{\emptyset \neq J\subseteq I } a_J \mathfrak{s}_J  $ such that $a_J \in \{0, 1\}$.
Finally, for a set $A \subseteq 2^I\setminus \{ \emptyset \} $, we write $\mathcal{F}^{-1}(A)$ for the hypergraphic polytopes $\mathfrak{q} = \sum_{J\in A} \mathfrak{s}_J $.
Note that a fundamental hypergraphic polytope is of the form $\mathcal{F}^{-1}(A)$ for some family $A \subseteq 2^{I} \setminus  \{ \emptyset \}$.

One can easily note that the hypergraphic polytope $\mathfrak{q} $ and $\mathcal{F}^{-1}(\mathcal{F}(\mathfrak{q})) $ are, in general, distinct, so some care will come with this notation.
However, the face structure is the same, and we give an explicit combinatorial equivalence in \cref{prop:nestosmplmodrel}.
If $\mathfrak{q} $ is a hypergraphic polytope such that $\mathcal{F}(\mathfrak{q}) $ is a building set, then $ \mathfrak{q}$ is called a \textit{nestohedron}, see \cite{pilaud17} and \cite{aguiar17}.
Hypergraphic polytopes and its subfamilies are studied in \cite[Part 4]{aguiar17}, where they are also called $y$-positive generalized permutahedra.

In \cite{aguiar17}, Aguiar and Ardila define \textbf{GP}, a Hopf algebra structure on the linear space generated by generalized permutahedra in $\mathbb{R}^n$ for $n\geq 0$.
The Hopf subalgebra \textbf{HGP} is the linear subspace generated by hypergraphic polytopes.
We warn the reader of the use of the same notation for Minkowski operations (Minkowski sum and dilations) and for algebraic operations in \textbf{GP}. However, the distinction should be clear from the context.
In \cite{doker11}, generalized permutahedra are also debated.

In \cite{grujic17}, Gruji{\'c} introduced a quasisymmetric map in generalized permutahedra $\Psi_{\mathbf{ GP }}: \mathbf{ GP } \to \mathit{QSym }$ that was extended to a weighted version in \cite{marko17}.
For a polytope $\mathfrak{q}\subseteq \mathbb{R}^I$, Gruji{\'c} defines a function $f: I \to \mathbb{N}$ to be $\mathfrak{q}$-generic if the face of $\mathfrak{q}$ that minimizes $\sum_{i\in I} f(i) x_i $, denoted $\mathfrak{q}_f$, is a point.
Equivalently, $f$ is $\mathfrak{q}$-generic if it lies in the interior of the normal cone of some vertex of $\mathfrak{q}$.
Then, Gruji{\'c} defines for a set $\{ x_i \}_{i\geq 1}$ of commutative variables, the quasisymmetric function:
\begin{equation}\label{eq:C1gperchrmor}
    \Psi_{\mathbf{ GP }}(\mathfrak{q}) = \sum_{f \text{ is }\mathfrak{q}\text{-generic} } x_f  \,.
\end{equation}

This quasisymmetric function is called the \textit{chromatic quasisymmetric function on generalized permutahedra}, or simply \textit{chromatic quasisymmetric function}.

We discuss now a non-commutative version of $\Psi_{\mathbf{GP}}$, where we establish an analogue of \cref{thm:graphkernel} for hypergraphic polytopes.
For that, consider the Hopf algebra of word quasisymmetric functions \textbf{WQSym}, an analogue of $QSym $ in non-commutative variables introduced in \cite{novelli06} that is also called non-commutative quasisymmetric functions, or \textbf{NCQSym}, for instance in \cite{bergeron09}.
For a generalized permutahedron $\mathfrak{q}$ and non-commutative variables $\{ \mathbf{a}_i \}_{i \geq 1}$, let $\mathbf{a}_f = \mathbf{a}_{f(1)} \cdots \mathbf{a}_{f(n)} $ and define
$$\boldsymbol{\Upsilon }_{\mathbf{GP}} (\mathfrak{q}) = \sum_{f \text{ is } \mathfrak{q} \text{-generic} } \mathbf{a}_f \, .$$

We see from \cref{prop:psiisquasisym} that $\boldsymbol{\Upsilon }_{\mathbf{GP}} (\mathfrak{q})$ is a word quasisymmetric function.
Moreover, a straightforward computation shows that $\boldsymbol{\Upsilon }_{\mathbf{GP}}$ defines a Hopf algebra morphism between \textbf{GP} and \textbf{WQSym}.
Let us call $\Psi_{\mathbf{HGP}} $ and $\boldsymbol{\Upsilon }_{\mathbf{HGP}}$ the restrictions of $\Psi_{\mathbf{GP}} $ and $\boldsymbol{\Upsilon }_{\mathbf{GP}}$ to \textbf{HGP}, respectively.

Our next theorems describe the kernel of the maps $\Psi_{\mathbf{HGP}} $ and $\boldsymbol{\Upsilon }_{\mathbf{HGP}}$, using two types of relations:
\begin{itemize}
\item the \textit{simple relations}, which are presented in \cref{prop:nestosmplmodrel}, and convey that $\uhsm_{\gpHa}(\mathfrak{q} )$ only depends on which coefficients $a_I $ are positive;

\item the \textit{modular relations}, which are exhibited in \cref{thm:nestomodrel}.
We note for future reference that these generalize the ones for graphs: some of the modular relations on hypergraphic polytopes are the image of modular relations on graphs by a suitable embedding map $Z$, introduced below.
\end{itemize}

\begin{thm}[Kernel and image of of $\boldsymbol{\Upsilon }_{\mathbf{HGP}} : \mathbf{ HGP } \to \mathbf{WQSym} $]\label{thm:nestokernel}
The space $\ker \boldsymbol{\Upsilon }_{\mathbf{HGP}} $ is generated by the simple relations and the modular relations on hypergraphic polytopes.
The image of $\boldsymbol{\Upsilon }_{\mathbf{HGP}} $ is $\mathbf{SC}$, a proper subspace of \textbf{WQSym} introduced in \cref{def:SCspace} below.
\end{thm}

Let us denote by $ \mathbf{WQSym}_n$ the linear space of homogeneous word quasisymmetric functions of degree $n$, and let $\mathbf{SC}_n = \mathbf{SC}\cap \mathbf{WQSym}_n$.
A monomial basis for \textbf{SC} is presented in \cref{def:SCspace}.
An asymptotic for the dimension of $\mathbf{SC}_n$ is computed in \cref{prop:growthdimim}, where in particular it is shown that it is exponentially smaller than the dimension of $ \mathbf{WQSym}_n$.

Two generalized permutahedra $\mathfrak{q}_1, \mathfrak{q}_2$ are isomorphic if $\mathfrak{q}_1$ can be obtained from $\mathfrak{q}_2$ by rearranging the coordinates of the points in $\mathfrak{q}_1$.
If $\mathfrak{q}_1, \mathfrak{q}_2$ are isomorphic, the chromatic quasisymmetric functions $\Psi_{\mathbf{GP}}(\mathfrak{q}_1 )$ and $ \Psi_{\mathbf{GP}}( \mathfrak{q}_2 )$ are the same.
We say that $\mathfrak{q}_1 - \mathfrak{q}_2 $ is an \textit{isomorphism relation on hypergraphic polytopes}.

\begin{thm}[Kernel and image of $\Psi_{\mathbf{HGP}} : \mathbf{ HGP } \to QSym $]\label{thm:nestocomukernel}
The linear space $\ker \Psi_{\mathbf{HGP}}$ is generated by the simple relations, the modular relations and the isomorphism relations.
The image of $\Psi_{\mathbf{HGP}} $ is $QSym$.
\end{thm}


In \cite{aguiar17}, Aguiar and Ardila define the graph zonotope, a Hopf algebra embeddimg $Z:\mathbf{G} \to \mathbf{GP}$ discussed above.
Remarkably, we have that $\Psi_{\mathbf{G}} \circ Z = \Psi_{\mathbf{GP}}$.
They also define other polytopal embeddings from other combinatorial Hopf algebras $\mathbf{h}$, like matroids, to \textbf{GP}.
One associates a universal morphism $\Psi_{\mathbf{h}}$ to these Hopf algebras that also satisfy $\Psi_{\mathbf{GP}} \circ Z = \Psi_{\mathbf{h}}$.
These universal morphisms are discussed below.

In particular, we can see that $Z ( \ker \Psi_{\mathbf{h}} ) = \ker\Psi_{\mathbf{GP}} \cap Z(\mathbf{h}) $.
This relation between $\ker \Psi_{\mathbf{h}}$ and $\ker \Psi_{\mathbf{GP}}$ is the main motivation to describe $\ker \Psi_{\mathbf{GP}}$, and indicates that $\ker \Psi_{\mathbf{GP}}$ is the kernel problem that deserves most attention.
In this paper, we leave the description of $\ker \Psi_{\mathbf{GP}} $ as an open problem.

Most of the combinatorial objects embedded in \textbf{GP} are also embedded in \textbf{HGP}, such as graphs and matroids, so a description of $\ker \Psi_{\mathbf{GP}}$ is already interesting.

We remark that a description of the generators of $\Psi_{\mathbf{GP}}$ or $\ker \Psi_{\mathbf{HGP}}$ does not entail a description of the generators of a generic $\ker \Psi_{\mathbf{h}}$.
For that reason, the kernel problem on matroids and on simplicial complexes is still open, despite these Hopf algebras being realized as Hopf subalgebras of $\mathbf{HGP}$.

On \cref{sec:facesgp}, the computation of the image in \cref{thm:nestokernel} is extended to the Hopf algebra of generalized permutahedra.
Specifically, there it is seen that the image of $\boldsymbol{\Upsilon}_{\mathbf{GP}}$ is also $\mathbf{SC}$.

\subsection*{Universal morphisms}

%


For a Hopf algebra $\mathbf{h}$, a \textit{character} $\eta $ of $\mathbf{h}$ is a linear map $\eta : \mathbf{h} \to \mathbb{K}$ that preserves the multiplicative structure and the unit of $\mathbf{h}$.
We define a \textit{combinatorial Hopf algebra} as a pair $(\mathbf{h}, \eta )$ where $\mathbf{h}$ is a Hopf algebra and $\eta : \mathbf{h} \to \mathbb{K}$ a character of $\mathbf{h}$.
For instance, consider the ring of quasisymmetric functions $QSym $ introduced in \cite{gessel84} with its monomial basis $\{M_{\alpha } \}$, indexed by compositions. 
Then, $QSym$ has a combinatorial Hopf algebra structure $(QSym, \eta_0 )$, by setting $\eta_0 ( M_{\alpha}) = 1$ whenever $\alpha $ has one or zero parts.

In \cite{aguiar06}, Aguiar, Bergeron, and Sottile showed that any combinatorial Hopf algebra $( \mathbf{h}, \eta )$ has a unique combinatorial Hopf algebra morphism $\Psi_{\mathbf{h}} : \mathbf{h} \to QSym $, i.e. a Hopf algebra morphism that satisfies $\eta_0\circ \Psi_{\mathbf{h}}  = \eta $.
In other words, $(QSym, \eta_0 )$ is a terminal object in the category of combinatorial Hopf algebras.
The construction of $\Psi_{\mathbf{h}}$ is given in \cite{aguiar06} and also presented below in \cref{ch2}.
We will refer to these maps as the universal maps to $QSym$.

The commutative invariants previously shown on graphs $\Psi_{\gHa } $, on posets $\Psi_{ \mathbf{Pos }} $ and on generalized permutahedra $\Psi_{\gpHa}$ can be obtained as universal maps to $QSym$.
If we take the character $\eta(G) = \mathbb{1}[G \text{ has no edges} ]$ on the graphs Hopf algebra, the unique combinatorial Hopf algebra morphism $\mathbf{G}\to QSym $ is exactly the map $\Psi_{\gHa} $.
With the Hopf algebra structure imposed on $\mathbf{GP} $ in \cite{aguiar17}, if we consider the character $\eta (\mathfrak{q}) = \mathbb{1}[\mathfrak{q} \text{ is a point} ]$, then $\Psi_{\mathbf{GP}}$ is the universal map from $\mathbf{GP}$ to $QSym$.
On posets, the Hopf algebra structure considered is the one presented in \cite{grinberg10} and the character that is considered is $\eta(P) = \mathbb{1}[ P \text{ is an antichain}]$.

To see the maps $\uhsm_{\gHa}:\gHa \to \mathbf{WSym} $ and $\uhsm_{\gpHa}:\gpHa \to \mathbf{WQSym} $ as universal maps, we need a parallel of the universal property of $QSym$ in the non-commutative world.
The fitting property is better described in the context of Hopf monoids in vector species.
Consider the Hopf monoid $\mathbf{\overline{WQSym } } $, which is presented in \cite{aguiar10} as the Hopf monoid of faces.
It is seen that there is a unique Hopf monoid morphism $\boldsymbol{ \Upsilon }_{\hs{h}}$ between a connected Hopf monoid $\hs{h} $ and $\hs{WQSym}$.
In the last chapter we establish another proof of this fact, using resources from character theory, and expand on that showing that instead of a connected Hopf monoid we can take any combinatorial Hopf monoid, for a suitable notion of combinatoric Hopf monoid.

The relationship between Hopf algebras and Hopf monoids is very well captured with the so called Fock functors, mapping Hopf monoids to Hopf algbras, and Hopf monoid morphisms to Hopf algebra morphism.
In particular, the full Fock functor $\mathcal{K}$ satisfies $\mathcal{K}(\mathbf{\overline{WQSym } })  =WQSym$.
Then, the universal property of $\mathbf{\overline{WQSym }}$ gives us a Hopf algebra morphism $\mathcal{K}(\uhsm_{\overline{h} } ) $ from $\mathcal{K}(\overline{h}) $ to $\mathbf{WQSym }$.
The maps $\boldsymbol{ \Upsilon}_{\gHa}, \boldsymbol{ \Upsilon}_{\gpHa} $ arise precisely in this way,  when applying $\mathcal{K}$ to the unique combinatorial Hopf monoid morphism from the Hopf monoid on graphs $\hs{G}$ and of generalized permutahedra $\hs{GP}$ to $\hs{WQSym} $.
In particular, we observe that $\mathcal{K}(\hs{G}) = \gHa$ and $\mathcal{K}(\hs{GP})= \gpHa$.
If we consider the poset Hopf monoid $\hs{Pos} $, the universal property of the combinatorial Hopf monoid $\hs{WQSym}$ gives us a non-commutative analogue $\uhsm_{\mathbf{Pos}}$ of the Gessel invariant, which coincides with the one presented in \cite{feray15}.
In particual, $\mathcal{K}(\hs{Pos}) = \mathbf{Pos}$.
We will refer to these Hopf algebra morphisms as the universal maps to $\mathbf{WQSym}$.

Finally, our previous results have an interesting consequence.
We show that, because $\uhsm_{\mathbf{HGP}}$ is not surjective, there is no combinatorial Hopf monoid morphism from the Hopf monoid on posets to the Hopf monoid on hypergraphic polytopes.
However, in \cite{aguiar17} a Hopf monoid morphism from posets to extended generalized permutahedra is constructed.
With this result we obtain that this map cannot be restricted from extended generalized permutahedra to generalized permutahedra.

\bigskip

Note: for sake of clarity, we have been using boldface for non-commutative Hopf algebras, their elements, and the associated combinatorial objects, like word symmetric functions and set compositions.
We try and maintain that notational convention throughout the paper.

This paper is organized as follows: In \cref{ch2} we address the preliminaries, where the reader can find the linear algebra tools that we use, the introduction to the main Hopf algebras of interest, and the proof that the several definitions of a generalized permutahedra are equivalent.
In \cref{sec:3} we prove \cref{thm:graphkernel,thm:graphcomukernel}, and we study the augmented chromatic invariant.
In \cref{ch4} we prove \cref{thm:nestokernel,thm:nestocomukernel}, and we present asymptotics for the dimension of the graded Hopf algebra \textbf{SC}.
In \cref{ch:HM} we present the universal property of $\hs{WQSym}$.
In \cref{app:coef} we find some relations between the coefficients of the augmented chromatic symmetric function and the coefficients of the original chromatic symmetric function on graphs.



\section{Preliminaries}\label{ch2}

There are natural maps $\mathbf{WSym} \to Sym $ and $\mathbf{WQSym} \to QSym $ by allowing the variables to commute.
We denote these maps by $comu $.

For an equivalence relation $\sim $ on a set $A$, we write $[ x]_{\sim}$ for the equivalence class of $x$ in $\sim $, and write $[x ]$ when $\sim $ is clear from context.
We write both $\mathcal{E}(\sim )$ and $A  / \sim $ for the set of equivalence classes of $\sim $.
All the vector spaces and algebras are over a generic field $\mathbb{K}$ of characteristic zero.

\subsection{Linear algebra preliminaries}

The following linear algebra lemmas will be useful to compute generators of the kernels and the images of $ \Psi$ and $ \boldsymbol{\Upsilon } $.
These lemmas describe a sufficient condition for a set $\mathcal{B}$ to span the kernel of a linear map $\phi:V\to W$.

\begin{lm}\label{lm:kernandimdesc}
Let $V$ be a finite dimensional vector space with basis $\{ a_i | i \in [m] \}$, $\phi: V \to W$ be a linear map, and $\mathcal{B} = \{b_j | j \in J \} \subseteq \ker \phi $ be a family of relations.

Assume that there exists $I \subseteq [m]$ such that:

\begin{itemize}
    \item the family $\{ \phi (a_i ) \}_{ i \in I }$ is linearly independent in $W $,
    
    \item for $i \in [m] \setminus I $ we have $a_i = b + \sum_{k=i+1}^m \lambda_{k, i} a_k $ for some $b \in \mathcal{B} $ and some scalars $\lambda_{k, i}$;
\end{itemize}

Then $\mathcal{B} $ spans $\ker \phi $.
Additionally, we have that $\{ \phi (a_i ) \}_{ i \in I }$ is a basis of the image of $\phi $.
\end{lm}


The following lemma will help us dealing with the composition $\Psi= \comu \circ \boldsymbol{\Upsilon } $: we give a sufficient condition for a natural enlargement of the set $\mathcal{B} $ to generate $\ker \Psi $, given that $\mathcal{B}$ already generates $\ker \boldsymbol{\Upsilon }$.

\begin{lm}\label{lm:kernofcomp}
We will use the same notation as in \cref{lm:kernandimdesc}.
Additionally, consider $\phi_1: W \to W' $ linear map and write $\phi' = \phi_1 \circ \phi $.
Take the equivalence relation $\sim$ in $\{ a_i\}_{ i \in [m]} $ that satisfies $a_i \sim a_j  $ whenever $ \phi'(a_i ) = \phi'(a_j )$.
Let $\mathcal{C} = \{ a_i- a_j | \, a_i \sim a_j \} $ and write $\phi' ([a_i] ) = \phi' (a_i )$ with no ambiguity.

\begin{equation}\label{cd:kerncomudiag}
\begin{tikzcd}
\mathcal{B}
\arrow[r, hook]
&
V
\arrow[r, "{\phi}"]
\arrow[rd, "{\phi'}"]
&
W
\arrow[d, "{\phi_1}"]
\\
\mathcal{C}
\arrow[ru, hook]
&
&
W'
\end{tikzcd}
\end{equation}

Assume the hypotheses in \cref{lm:kernandimdesc} and, additionally, suppose that the family $\{ \phi' ([a_i] ) \}_{ [a_i] \in \mathcal{E}( \sim ) } $ is linearly independent in $W'$.

Then, $\ker \phi' $ is generated by $\mathcal{B} \cup \mathcal{C}$.
Furthermore, $\{ \phi' ([a_i] ) \}_{ [a_i] \in \mathcal{E}( \sim ) } $ is a basis of $\im \phi ' $.
\end{lm}

\begin{proof}[Proof of \cref{lm:kernandimdesc}]
Suppose, for sake of contradiction, that there is some element $c \in \ker \phi \setminus \spn\mathcal{B} $.
In particular $c \neq 0$.
Write
\begin{equation}\label{eq:binabasis}
c = \sum_{k=1 }^m \tau_k a_k \, ,
\end{equation}
and note that if $\tau_i = 0$ for every $i \not \in I$, then
$$ \, 0 = \phi(c) = \phi \left(  \sum_{k \in I } \tau_k a_k \right) = \sum_{k \in I } \tau_k \phi ( a_k) , $$
which, by linear independence of $\{ \phi ( a_k)\}_{k \in I }$, implies that $\tau_k = 0$ for every $k \in I$, contradicting $c \neq 0$.
Therefore, we have $\tau_i \neq 0 $ for $i \not\in I$ whenever $c\in \ker \phi \setminus \spn\mathcal{B}$.

Consider the smallest index $i_c \in [m]\setminus I $  such that $\tau_i $ is non-zero.
Consider $c \in \ker \phi \setminus \langle\mathcal{B} \rangle $ that maximizes $i_c$.

Thus, we can write
\begin{equation}\label{eq:spanofb}
c = \sum_{j\in I } \tau_j a_j +  \sum_{\substack{j\in [m] \setminus I \\ j \geq i_c }} \tau_j a_j \, .
\end{equation}

By hypotheses, because $i_c\not\in I$, there is some $b' \in \mathcal{B}$ such that:
$$ a_{i_c} = b' + \sum_{j=i_c+1}^m \lambda_{j, i_c} a_j  \, .$$

So applying this to \eqref{eq:spanofb} gives us:
\begin{equation}\begin{split}
c - \tau_{i_c} b' &= \sum_{j\in I } \tau_j a_j +  \sum_{\substack{j\in [m] \setminus I \\ j \geq i_c }} \tau_j a_j - \tau_{i_c} a_{i_c} + \sum_{k= i_c+ 1}^m \tau_{i_c} \lambda_{k, i_c} a_k \\
                  &= \sum_{j\in I } \tau_j a_j +  \sum_{\substack{j\in [m] \setminus I \\ j > i_c }} \tau_j a_j + \sum_{j= i_c+ 1}^m \tau_{i_c} \lambda_{j, i_c} a_j \, .
\end{split}
\end{equation}
Note that $c - \tau_{i_c} b_j \in \ker \phi \setminus \spn \mathcal{B} $ which contradicts the maximality of $i_c$.
From this we conclude that there are no elements $c$ in $ \ker \phi \setminus \spn \mathcal{B} $.

To show that the family $\{\phi(a_i) \}_{i\in I} $ is a basis of $\im \phi$, we just need to establish that this is a generating set.
Naturally, $\{\phi(a_i)\}_{i \in I} \cup \{\phi(a_i)\}_{i\in [m]}$ is a generating set because it is the image of a basis of $V$.
We show by induction that $\{\phi(a_i)\}_{i \in I} \cup \{\phi(a_i)\}_{i\in [m]\setminus [k]}$ is a generating set for any non-negative $k\leq m+1$.
This concludes the proof, since the original claim is this for $k=m+1$.

Indeed, if $\{\phi(a_i)\}_{i \in I} \cup \{\phi(a_i)\}_{i\in [m]\setminus [k]}$ is a generating set of $\im \phi$, then we note $a_k = b + \sum_{j=k+1}^m \lambda_{j, k} a_j $ for some $b \in \mathcal{B} $, so 
$$\phi(a_k) \in \spn \{\phi(a_i)\}_{i \in I} \cup \{\phi(a_i)\}_{i\in [m]\setminus [k+1]}\, . $$
This concludes the induction step.
\end{proof}

\begin{proof}[Proof of \cref{lm:kernofcomp}]
Define $I' = \{ \max A \cap  I | \, A \in \mathcal{E}(\sim ) \} $.
Note that for every $j \in I \setminus I' $ there is $c \in \mathcal{C}$ such that $a_j = c + \sum_{k=j+1}^m \lambda_{k, j} a_k$.
Indeed it is enough to choose $i \sim j $ with $i \in I' $, to write $a_j = \underbrace{a_j - a_i}_{\in \mathcal{C} } + a_i $.

So, the set $I' \subseteq [m] $ satisfies both that:

\begin{itemize}
    \item We have by hypothesis that $\{ \phi' (a_i ) \}_{ i \in I' } = \{ \phi' ([a_i] ) \}_{ i \in I'}$ is linearly independent in $W'$;
    
    \item For $i \in [m] \setminus I' $ we can write $a_i = b + \sum_{k=j+1}^m \lambda_{k, i} a_k $ for some $b \in \mathcal{B} \cup \mathcal{C} $ and some scalars $\lambda_{k, i}$.
\end{itemize}

Now applying \cref{lm:kernandimdesc} to $I'$ instead of $I$, to $\phi' $ instead of $\phi $ and to $\mathcal{B}\cup \mathcal{C} $ instead of $\mathcal{B} $ tells us that $\mathcal{B}\cup \mathcal{C} $ generates $\ker \phi' $, and that $\{ \phi' ([a_i] ) | i \in I' \} = \{ \phi' (a_i ) | i \in I \}  $ spans the image of $\phi'$, as desired.
\end{proof}

\subsection{Hopf algebras and associated combinatorial objects}
\label{sec:halg}

In the following, all the Hopf algebras $\mathbf{H}$ have a grading, denoted by $\mathbf{H}= \oplus_{n \geq 0 } \mathbf{H}_n $.

An \textit{integer composition}, or simply a composition, of $n$, is a list $\alpha = (\alpha_1, \cdots , \alpha_k)$ of positive integers whose sum is $n$.
We write $\alpha \models n$.
We denote the length of the list by $l(\alpha )$ and we denote the set of compositions of size $n$ by $\mathcal{C}_n$.

An \textit{integer partition}, or simply a partition, of $n$, is a non-increasing list of positive integers $\lambda = ( \lambda_1, \cdots , \lambda_k )$ whose sum is $n$.
We write $\lambda \vdash n$.
We denote the length of the list by $l(\lambda ) $ and we denote the set of partitions of size $n$ by $\mathcal{P}_n$.
By disregarding the order of the parts on a composition $\alpha $ we obtain a partition $\lambda ( \alpha ) $.

A \textit{set partition} $\parpi = \{ \parpi_1, \cdots, \parpi_k\}$ of a set $I$ is a collection of non-empty disjoint subsets of $I$, called \textit{blocks}, that cover $I$.
We write $\parpi \vdash I$.
We denote the number of parts of the set partition by $l(\parpi ) $, and call it its length.
We denote the family of set partitions of $I$ by $\mathbf{P}_I$, or simply by $\mathbf{P}_n$ if $I = [n] $.
By counting the elements on each block of $\parpi$, we obtain an integer partition denoted by $\lambda(\parpi ) \vdash \# I $.
We identify a set partition $\parpi\in \mathbf{P}_I$ with an equivalence relation $\sim_{\parpi}$ on $I$, where $x \sim_{\parpi} y $ if $x, y\in I$ are on the same block of $\parpi$.

A \textit{set composition} $\opi = S_1| \cdots | S_l $ of $I$ is a list of non-empty disjoint subsets of $I$ that cover $I$, which we call \textit{blocks}.
We write $\opi  \models I$.
We denote the size of the set composition by $l(\opi ) $.
We write $\oPi_I$ for the family of set compositions of $I$, or simply $\oPi_n$ if $I = [n] $.
By disregarding the order of a set composition $\opi$, we obtain a set partition $\makepar (\opi) \vdash I$.
By counting the elements on each block of $\opi$, we obtain a composition denoted by $\alpha (\opi ) \models \# I $.
A set composition is naturally identified with a total preorder $R_{\opi }$ on $I$, where $x R_{\opi } y$ if $x\in S_i, y\in S_j$ for $i \leq j$.

Permutations act on set compositions and set partitions: for a set composition $\opi = ( S_1, \cdots ,  S_k )$, a set partition $\pi = \{ \pi^{(1)}, \cdots , \pi^{(k)} \}$ on $I$, and a permutation $\phi: I \to I$, we define the set composition $\phi( \opi ) =  ( \phi(S_1), \cdots ,  \phi(S_k) )$ and the set partition $\phi(\parpi ) = \{ \phi(\parpi^{(1)}), \cdots , \phi(\parpi^{(k)}) \} $.

A \textit{coloring} of the set $I$ is a function $f : I \to \mathbb{N} $.
The set composition type $\opi(f)$ of a coloring $f: I \to \mathbb{N} $ is the set composition obtained after deleting the empty sets of $f^{-1}(1) | f^{-1}(2) |  \cdots $.
This notation is extended to function $f:I \to \mathbb{R}$.

In partitions and in set partitions, we use the classical \textit{coarsening orders} $\leq $ with the same notation, where we say that $\pi \leq \tau $ (resp. $\parpi \leq \partau $) if $\tau $ is obtained from $\pi$ by adding some parts of the original parts together (resp. if $\partau $ is obtained from $\parpi $ by merging some blocks).

These objects relate to the Hopf algebras $Sym$, $QSym$, $\mathbf{WSym}$ and $\mathbf{WQSym}$.
The homogeneous component $Sym_n$ (resp. $QSym_n$, $\mathbf{WSym}_n$ and $\mathbf{WQSym}_n$) of the Hopf algebra $Sym$ (resp. $QSym$, $\mathbf{WSym}$, $\mathbf{WQSym}$)
has a monomial basis indexed by partitions (resp. compositions, set partitions, set compositions), which we denote by $\{m_{\lambda } \}_{\lambda \in \mathcal{P}_n}$ (resp. 
$\{M_{\alpha } \}_{\alpha \in \mathcal{C}_n} $, 
$\{\mathbf{m}_{\parpi } \}_{\parpi \in \mathbf{P}_n}$ and 
$\{\mathbf{M}_{\opi } \}_{\opi \in \oPi_n}$).

\subsection{Hopf algebras on graphs and posets\label{sec:GPos}}

Of interest are the Hopf algebras on graphs $\mathbf{G}$ and on posets $\mathbf{Pos}$, which are graded and connected, and whose homogeneous components  $\mathbf{G}_n$, resp. $\mathbf{Pos}_n$, are the linear span of the graphs with vertex set $[n]$, resp. partial orders in the set $[n]$.

In these graded vector spaces, define the products and coproducts in the basis elements.
For that, when $A, B$ are sets of integers with the same cardinality, we let $rl_{A, B}$ be the canonical relabelling of combinatorial objects on $A$ to combinatorial objects on $B$ that preserves the order of the labels.

Recall that the disjoint union of graphs $(V_1, E_1), (V_2, E_2)$, where $V_1 \cap V_2 = \emptyset$, is $(V_1 \sqcup V_2, E_1\sqcup E_2)$, and the restriction of a graph $G|_I$ is $(I, E(G)\cap \binom{I}{2})$.
Denote $[m] = \{ 1, \dots , m\}$ as usual, and $[m, n] =\{m, m+1, \dots , n \}$ for $n \geq m$.
Given two graphs $G_1, G_2$ with vertices labeled in $[n], [m]$ respectively, the product is the relabeled disjoint union
$$G_1 \cdot G_2 = G_1 \sqcup rl_{[m], [m+1, m+n]}(G_2)\, .$$
For the coproduct, let $G$ be a graph labeled in $[n]$, then 
$$\Delta G = \sum_{[n] = I\sqcup J } rl_{I, [\# I ] }(G|_I )  \otimes rl_{J, [\# J ] }(G|_J )  \, . $$

To define a Hopf algebra on posets, consider two posets $P_1 = (S_1, R_{P_1}), P_2 = (S_2, R_{P_2})$, where $R$ represents the set of pairs $(x, y)$ such that $x \leq y$ in the respective poset.
The disjoint union of posets is written $P_1 \sqcup P_2$ and defined as $(S_1\sqcup S_2, R_{P_1} \sqcup R_{P_2})$, and the restriction of a poset $P= (S, R_P)$ is written $P|_I$ and defined as $(I, R_P\cap (I \times I))$.
Recall that $S$ is an ideal of $P$ if whenever $x\leq y $ and $x\in S$, then $y\in S$.
Define the product between partial orders $P, Q$ in the sets $[n], [m]$, respectively, as
$$ P \cdot Q = P \sqcup rl_{[n], [n+1, n+m]}(Q) \, ,$$
and the coproduct for a partial order $P$ in $[n]$.
$$\Delta P = \sum_{S \text{ ideal of } P } rl_{S, [\# S]}(P|_S ) \otimes rl_{S^c, [n - \# S] }(P|_{S^c})  \, . $$
These operations define a Hopf algebra structure in $\mathbf{G} $ and $\mathbf{Pos}$, as described in \cite{grinberg10}.

Recall from the introduction that, for graphs, Gebhard and Sagan defined in \cite{gebhard99} the non-commutative chromatic morphism.
The following expression is given:

\begin{lm}[\citestan]\label{lm:monbasis}
For a graph $G$ we say that a set partition $\partau $ of $V(G)$ is proper if no block of $\partau $ contains an edge.
Then have that 
$$\boldsymbol{\Upsilon }_{\mathbf{G}} (G) = \sum_{\partau } \mathbf{m}_{\partau}\, ,$$
where the sum runs over all proper set partitions of $V(G)$.
\end{lm}

\subsection{Faces  and a Hopf algebra structure of generalized permutahedra\label{sec:2gp}}

In the following we identify $I$ with $[n]$.
For a set composition $\opi = S_k|\dots | S_1 $ on $[n]$, recall that $R_{\opi}$ is a partial order on $[n]$.
For a non-empty set $J \subseteq [n] $, define the set $J_{\opi} = \{ \text{minima of J in } R_{\opi} \} = J\cap S_i$, where $i$ is the smallest index with $J \cap S_i \neq \emptyset $.
A coloring on $[n]$ is a map $f:[n]\to \mathbb{N}$.
A real coloring on $[n]$ is a map $f:[n]\to \mathbb{R}$, and we identify the real coloring $f$ with the linear function $f:\mathbb{R}^{[n]} \to \mathbb{R}$.
$$ x \mapsto \sum_{i=1}^n f(i)x_i \, .$$

In the space $\mathbb{R}^{[n]}$, we define the simplices $\mathfrak{s}_J  = \conv \{e_v| \, v \in J \} $ for each $J\subseteq [n]$.
Recall that a \textit{generalized permutahedron} is a Minkowski sum and difference of the form
$$\mathfrak{q} = \left( \sum_{\substack{J\neq \emptyset \\ a_J > 0 }} a_J \mathfrak{s}_J \right) - \left( \sum_{\substack{J\neq \emptyset \\ a_J < 0 }} |a_J| \mathfrak{s}_J \right) \, ,$$
for reals $\mathcal{L}(\mathfrak{q}) = \{a_J \}_{\emptyset \neq J \subseteq [n]}$ that can be either positive, negative or zero.

Recall as well that a \textit{hypergraphic polytope} is a generalized permutahedron of the form
$$\mathfrak{q} =  \sum_{J\neq \emptyset} a_J \mathfrak{s}_J  \, ,$$
for non-negative reals $\mathcal{L}(\mathfrak{q}) = \{a_J \}_{\emptyset \neq J \subseteq [n]}$.

For a polytope $\mathfrak{q}$ and a real coloring $f$ on $[n]$, we denote by $\mathfrak{q}_f$ the subset of $\mathfrak{q}$ on which $f$ is minimized, that is
$$\mathfrak{q}_f := \arg \min_{x\in\mathfrak{q}} \sum_{i\in I} f(i)x_i \, . $$
A face of $\mathfrak{q}$ is the solution to such a linear optimization problem on $\mathfrak{q}$.
A real coloring is said to be $\mathfrak{q}$-generic if the corresponding face is a point.

\begin{smpl}
Consider the hypergraphic polytope $\mathfrak{q} = \mathfrak{s}_{\{1, 2, 3\}} + \mathfrak{s}_{\{1, 2\}} $ in $\mathbb{R}^3$.
If we take the coloring of $\{1, 2, 3\}$ given by $f(1) = f(2) = 1$ and $f(3) = 3$, then $\mathfrak{q}_f = \mathfrak{s}_{\{1, 2\}} $.
If we consider the coloring $g(1) = g(3) = 2 $ and $g(2) = 1  $, then $  \mathfrak{q}_g = 2 \mathfrak{s}_{\{ 2\}}  $ is a point, so $g$ is $\mathfrak{q}$- generic.
\end{smpl}

In particular, note that if $J_1 \subseteq J_2$, then $\mathfrak{s}_{J_1}$ is a face of $ \mathfrak{s}_{J_2}$.
Incidentally, whenever $f$ is a coloring that is minimal exactly in $J_1$, we have that $\mathfrak{s}_{J_1} = (\mathfrak{s}_{J_2})_f$ .
In fact, for a real coloring $f: [n] \to \mathbb{N}$ the face corresponding to $f$ of a simplex is another simplex, specifically it we can directly compute that
\begin{equation}\label{eq:facesimpl}
(\mathfrak{s}_J)_f = \mathfrak{s}_{J_{\opi (f)}}\, .
\end{equation}

The following fact describes faces of the Minkowski sums and differences:

\begin{lm}\label{lm:minkfaces}
Let $f$ be a  real coloring and $\mathfrak{a}, \mathfrak{b} $ two polytopes.
Then $(\mathfrak{a}+ \mathfrak{b})_f = \mathfrak{a}_f + \mathfrak{b}_f$ and, if the difference $\mathfrak{a} - \mathfrak{b}$ is well defined, $(\mathfrak{a}- \mathfrak{b})_f = \mathfrak{a}_f - \mathfrak{b}_f$.
\end{lm}

\begin{proof}
Suppose that $m_{\mathfrak{a}}, m_{\mathfrak{b}} $ are the minima of $f$ in the polytopes $\mathfrak{a}, \mathfrak{b}$.
Let $x \in \mathfrak{a} + \mathfrak{b}$.
So $x= a+b $ for some $a\in\mathfrak{a}, b\in\mathfrak{b}$.

Then $f(x) = f(a) + f(b) \geq m_{\mathfrak{a}} + m_{\mathfrak{b}}$.
We have equality if and only if we have $ a\in \mathfrak{a}_f, b\in\mathfrak{b}_f$, that is when $x\in \mathfrak{a}_f +\mathfrak{b}_f$.

Now $(\mathfrak{a}- \mathfrak{b})_f = \mathfrak{a}_f - \mathfrak{b}_f$ follows because $(\mathfrak{a}- \mathfrak{b})_f + \mathfrak{b}_f= \mathfrak{a}_f$ by the above.
\end{proof}

\begin{defin}[Normal fan of a polytope]
A cone is a subset of an $\mathbb{R}$-vector space that is closed for addition and multiplication by positive scalars.
For a polytope $\mathfrak{q} $ and $F\subseteq \mathfrak{q}$ one of its faces, we define its normal cone
$$\mathcal{N}_{\mathfrak{q}}(F):= \{f:[n] \to \mathbb{R} | \, \, \mathfrak{q}_f = F \} \, .$$

This is a cone in the dual space of $\mathbb{R}^n $.
Moreover, the normal cones of all the faces of $\mathfrak{q} $ partition $(\mathbb{R}^n )^*$ into cones $\mathcal{N}_{\mathfrak{q}}= \{\mathcal{N}_{\mathfrak{q}}(F) | \, F \text{ is a face of } \mathfrak{q} \}$.
This is the \textit{normal fan} of $\mathfrak{q}$.
\end{defin}

\begin{smpl}[The normal fan of the $n$-permutahedron - The braid fan]\label{smpl:permnc}
The faces of the permutahedron are indexed by $\oPi_n$.
In particular, the corresponding normal cone of the face $F_{\opi} $, corresponding to $\opi \in \oPi_n $, is 
$$\mathcal{N}(F_{\opi} ) = \{f: [n] \to \mathbb{R} | \opi(f) = \opi \} \, . $$
\end{smpl}

In the introduction we referred two other definitions of generalized permutahedra that are present in the literature.
We recover them here, and justify their equivalence:

\begin{lm}[Definition 1 of generalized permutahedra, see \cite{aguiar17}]\label{lm:normfan}
A polytope $\mathfrak{q}$ is a generalized permutahedron in the sense of \eqref{eq:gpcoefs} if and only if its normal fan coarsens the one of the permutahedron.
Specifically, for any two real colorings $f_1, f_2$, if $\opi(f_1)=\opi(f_2)$ then $\mathfrak{q}_{f_1} = \mathfrak{q}_{f_2}$.
\end{lm}

Define the polytope $\mathcal{P}_n^z(\{z_I\}_{\emptyset \neq I\subseteq [n]} )$ in the plane $\sum_i x_i = z_{[n]}$ given by the inequalities
$$ \sum_{i\in I} x_i \geq z_I \, , $$
for some real numbers $\{ z_I\}_{\emptyset \neq I\subseteq [n] }$.

\begin{lm}[Definition 2 of generalized permutahedra, see \cite{postnikov09}]\label{lm:facets}
A polytope is a generalized permutahedron if it can be expressed as $\mathcal{P}_n^z(\{z_I\}_{\emptyset \neq I\subseteq [n]} )$ for real numbers $\{ z_I\}_{I\subseteq [n] }$ such that 
$$ z_I + z_J \leq z_{I\cup J} + z_{I\cap J} \, , $$
for all non-empty sets $I, J\subseteq [n]$ that are not disjoint.
\end{lm}

In \cite[Theorem 12.3]{aguiar17}, is it shown that these two last notions of generalized permutahedra are equivalent.
That is, a polytope $\mathfrak{q}$ is of the form $\mathfrak{q} = \mathcal{P}_n^z(\{z_I\}_{\emptyset \neq I\subseteq [n]} )$ for real numbers $\{ z_I\}_{\emptyset \neq I\subseteq [n] }$ if and only if its normal fan coarsens the one from the permutahedron.

In \cite[Proposition 2.4]{ardila10}, Ardila, Benedetti and Doker show that any generalized permutahedron has an expression of the from given by \cref{eq:gpcoefs}.
The main feature in that proof is the following:
for real numbers $\{ z_I\}_{\emptyset \neq I\subseteq [n] }$ such that $ z_I + z_J \geq z_{I\cup J} + z_{I\cap J}$, if we choose reals $\{a_J\}_{\emptyset\neq J\subseteq [n]}$ such that $z_I = \sum_{\emptyset \neq J\subseteq I} a_J $, then \cref{eq:gpcoefs} gives us a well defined polytope and in fact defines the same polytope as $\mathcal{P}_n^z(\{z_I\}_{\emptyset \neq I\subseteq [n]} )$.

In the following we establish that the normal fan of a polytope of the form \cref{eq:gpcoefs} coarsens the one of the $n$-permutahedron, concluding with the above that the three definitions of generalized permutahedra presented are equivalent.

\begin{prop}\label{prop:gperequiv}
Let $\mathfrak{q} $ be a polytope of the form
$$ \mathfrak{q} = \left( \sum_{J \in A_+ } a_J \mathfrak{s}_J \right) - \left( \sum_{J\in A_-} |a_J| \mathfrak{s}_J \right) \, , $$
for reals $\mathcal{L}(\mathfrak{q}) = \{a_J \}_{\emptyset \neq J \subseteq [n]}$ that can be either positive, negative or zero, and $A_+ = \{J |   a_J>0 \} $ and $ A_- = \{J |   a_J<0 \}  $.
Then its normal fan coarsens the one of the permutahedron.
\end{prop}

\begin{proof}
Let $f$ be a real coloring of $[n]$.
As a consequence of \cref{smpl:permnc} and as discussed in the end of \cref{lm:normfan}, it is enough to establish that if $f_1, f_2$ satisfy $\opi(f_1) = \opi(f_2) $ then $\mathfrak{q}_{f_1} = \mathfrak{q}_{f_2} $.
In fact, from \cref{lm:minkfaces} and \eqref{eq:facesimpl}, we have
$$ \mathfrak{q}_{f}\! =\! \left( \sum_{J \in A_+ } a_J (\mathfrak{s}_J)_{f} \right)\! -\! \left( \sum_{J\in A_-} |a_J| (\mathfrak{s}_J)_{f} \right)\!\! =\!\!  \left( \sum_{J \in A_+ } a_J \mathfrak{s}_{J_{\opi(f)}} \right)\! -\! \left( \sum_{J\in A_-} |a_J| \mathfrak{s}_{J_{\opi(f)}} \right) \, , $$
which clearly only depend on the set composition type of the real coloring $f$.
\end{proof}

Denote by $\mathfrak{q}_{\opi}$ the face on $\mathfrak{q}$ that is the solution to any linear optimization problem on $\mathfrak{q}$ for a real coloring $f$ with composition type $\opi(f) = \opi$, so that 
\begin{equation}\label{eq:facepermutahedra}
\mathfrak{q}_{\opi } = \left( \sum_{J \in A_+ } a_J \mathfrak{s}_{J_{\opi}} \right) - \left( \sum_{J\in A_-} |a_J| \mathfrak{s}_{J_{\opi}} \right)  \, .
\end{equation}
The following is a consequence of \cref{lm:normfan}:

\begin{prop}\label{prop:psiisquasisym}
If $\mathfrak{q} $ is a generalized permutahedron, then 
\begin{equation}\label{eq:uhsmiswqsym}
 \uhsm_{\gpHa } (\mathfrak{q}) = \sum_{f \, \mathfrak{q}\text{-generic }} \mathbf{a}_f   = \sum_{\mathfrak{q}_{\opi} = \pt } \mathbf{M}_{\opi }  \in \mathbf{WQSym}_n \, .
\end{equation}
\end{prop}

We now turn away from the face structure of generalized permutahedra and debate its Hopf algebra structure, introduced in \cite{aguiar17}.
As usual, consider a generalized permutahedron $\mathfrak{q}$ given by \eqref{eq:gpcoefs}.
If $\opi = A | B$ is a set composition of $[n]$, then $ \mathfrak{q}_{\opi }$ can be written as a Minkowski sum of polytopes
$$ \mathfrak{q}_{\opi } =: \mathfrak{q}|_A  + \mathfrak{q}\backslash_A \, ,$$
where $\mathfrak{q}|_A $ is a generalized permutahedron in $\mathbb{R}^A $ and $\mathfrak{q}\backslash_A $ is a generalized permutahedron on $\mathbb{R}^B $.
Note that $B = A^c $ so the dependence of $\mathfrak{q}|_A $ and $\mathfrak{q}\backslash_A $ on $B$ is implicit.

We can obtain explicit expressions for $\mathfrak{q}|_A $ and $\mathfrak{q}\backslash_A $:
$$\mathfrak{q}|_A\!  =\!  \left( \sum_{\substack{J\in A_+ \\ J\not\subseteq B }} a_J \mathfrak{s}_{J\cap A} \right) - \left( \sum_{\substack{J \in A_- \\J\not\subseteq B }} |a_J| \mathfrak{s}_{ J\cap A }\right) \, , \mathfrak{q}\backslash_A\!  =\!  \left( \sum_{\substack{J\in A_+ \\ J\subseteq B}} a_J \mathfrak{s}_J \right) - \left( \sum_{\substack{J\in A_- \\J\subseteq B  }} |a_J| \mathfrak{s}_J \right) \, .$$

We have now all the material to endow the space of generalized permutahedra with a Hopf algebra structure according to \cite{aguiar17}:
let $\mathbf{GP} = \oplus_{n\geq 0} \mathbf{GP}_n$, where $\mathbf{GP}_n $ is the free linear space on generalized permutahedra in $\mathbb{R}^n$.

The $\mathbf{GP }$ linear space has the following product, when $\mathfrak{q}_1, \mathfrak{q}_2$ are generalized permutahedra in $\mathbb{R}^n, \mathbb{R}^m$ respectively:
$$ \mathfrak{q}_1 \cdot \mathfrak{q}_2 = \mathfrak{q}_1 \times rl_{[m], [m+1, m+n]}(\mathfrak{q}_2 ) \, .$$

The $\mathbf{GP}$ linear space has the following coproduct, when $\mathfrak{q} $ is a generalized permutahedron in $\mathbb{R}^n $:
$$ \Delta \mathfrak{q} = \sum_{A \subseteq [n]} rl_{A, [\# A ]}(\mathfrak{q}|_A ) \otimes rl_{A^c, [n - \# A]} (\mathfrak{q}\backslash_A) \, . $$

\begin{rem}
Note that the span of the fundamental hypergraphic polytopes does not form a Hopf algebra, as it is not stable for the coproduct.
\end{rem}

\section{Main theorems on graphs\label{sec:3}}
In this section we prove \cref{thm:graphkernel,thm:graphcomukernel}, which follow from \cref{lm:kernandimdesc,lm:kernofcomp}.
We also discuss an application of \cref{thm:graphcomukernel} on the tree conjecture, by constructing a new graph invariant $\tilde{\Psi } (G )$ that satisfies the modular relations.

For a set partition $\parpi $, we define the graph $K_{\parpi}$ where $\{ i, j\} \in E(K_{\parpi}) $ if $i \sim_{\parpi } j$.
This graph is the disjoint union of the complete graphs on the blocks of $\parpi$.
We denote the complement of $K_{\parpi }$ as $K_{\parpi}^c $.
Note that a set partition $\partau$ is proper in $K_{\parpi}^c$ if and only if $\partau \leq \parpi $ in the coarsening order on set partitions.
Hence, 
as a consequence of \cref{lm:monbasis},
\begin{equation}\label{eq:fundgraphcase}
\boldsymbol{\Upsilon }_{\gHa } (K_{\parpi}^c ) = \sum_{\partau \leq \parpi} \mathbf{m}_{\partau }.
\end{equation}

We now show that the kernel of $\boldsymbol{\Upsilon }_{\mathbf{G}} $ is spanned by the modular relations.

\textit{Proof of \cref{thm:graphkernel}}.
Recall that $\mathbf{G}_n$ is spanned by graphs with vertex set $[n]$.
We choose an order $\tilde{\geq } $ in this family of graphs in a way that the number of edges is non-decreasing.

From \eqref{eq:fundgraphcase}, we know that the transition matrix of $\{\boldsymbol{\Upsilon }_{\mathbf{G}}(K_{\parpi }^c ) | \parpi \in \mathbf{P}_n  \}$ over the monomial basis of \textbf{WSym} is upper triangular, hence forms a basis set of $\mathbf{WSym}$.
In particular, $\im \uhsm_{\gHa} = \mathbf{WSym }$.

In order to apply \cref{lm:kernandimdesc} to the set of modular relations on graphs,
it suffices to show the following:
if a graph $G$ is not of the form $K^c_{\parpi} $, then we can find a formal sum $ G - G\cup\{e_1\} - G\cup \{ e_2\} + G\cup\{e_1, e_2\}$ that is a modular relation.
Indeed, $G$ is the graph with least edges in that expression, so it is the smallest in the order $\tilde{\geq } $.
It follows from \cref{lm:kernandimdesc} that the modular relations generate the space $\ker \boldsymbol{\Upsilon }_{\mathbf{G}} $.

To find the desired modular relation, it is enough to find a triangle $\{e_1, e_2, e_3 \}$ such that $e_1, e_2 \not \in E(G)$ and $e_3 \in E (G)$.
Consider $\partau $, the set partition given by the connected components of $G^c$, so that $G \supseteq K^c_{\partau}$.
By hypothesis, $G \neq K^c_{\partau }$, so there are vertices $u, w $ in the same block of $\partau $ that are not neighbors in $G^c $.
Without loss of generality we can take such $u, w$ that are at distance 2 in $G^c$, so they have a common neighbor $v$ in $G^c $ (see example in \cref{fig:gtp}).

\begin{figure}[ht]
\centering
\includegraphics[width=0.45\textwidth]{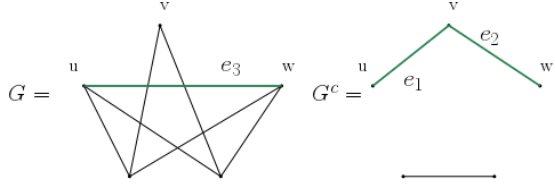}
\caption{\label{fig:gtp}Choice of edges in proof of \cref{thm:graphkernel}}
\end{figure}

The edges $e_1 = \{v, u \}$, $e_2 = \{v, w \} $ and $e_3 = \{u, w \}$ form the desired triangle, concluding the proof.
\qed

\begin{proof}[Proof of \cref{thm:graphcomukernel}]
It is clear that $\Psi_{\mathbf{G}} $ is surjective, since $\uhsm_{\gHa } $ is surjective.
Now, our goal is to apply \cref{lm:kernofcomp} to the map $\Psi_{\mathbf{G}} = \comu \circ \boldsymbol{\Upsilon }_{\mathbf{G}}$ 
and the equivalence relation corresponding to graph isomorphism.
First, note that if $\lambda(\parpi ) = \lambda(\partau) $ then $K_{\parpi}^c$ and $K_{\partau}^c $ are isomorphic graphs.
Define without ambiguity $ r_{\lambda (\parpi )} =  \Psi_{\gHa} ( K^c_{ \parpi} )$.

From the proof of \cref{thm:graphkernel}, the hypotheses of \cref{lm:kernandimdesc} are satisfied.
Therefore, to apply \cref{lm:kernofcomp} it is enough to establish that the family $\{ r_{\lambda }  \}_{\lambda \in \mathcal{P}_n }$ is linearly independent.
Indeed, it would follow that $\ker \Psi_{\mathbf{G}} $ is generated by the modular relations and the isomorphism relations, and $\{ r_{\lambda }  \}_{\lambda \in \mathcal{P}_n }$  is a basis of $\im \Psi_{\mathbf{G}}$ concluding the proof.



Recall that for set partitions $\parpi_1, \parpi_2$ we have that $\parpi_1 \leq \parpi_2 \Rightarrow \lambda(\parpi_1) \leq \lambda(\parpi_2) $.
The linear independence of $\{ r_{\lambda } \}_{\lambda \in \mathcal{P}_n }$ follows from the fact that its transition matrix to the monomial basis is upper triangular under the coarsening order in integer partitions.
Indeed, from \eqref{eq:fundgraphcase}, if we let $\partau$ run over set partitions and $\sigma $ run over integer partitions, we have
$$r_{\lambda (\parpi )} = \Psi_{\gHa}(K_{\parpi}^c ) = \sum_{\partau \leq \parpi} m_{\lambda (\partau )}  =  \sum_{\sigma \leq \lambda(\parpi )} a_{\parpi, \sigma }\,\, m_{\sigma} \, ,$$
where $a_{\parpi, \sigma} = \# \{\partau \vdash [n]  | \lambda(\partau ) = \sigma , \, \partau \leq \parpi  \} $.
Note that $a_{\parpi, \lambda ( \parpi )} = 1$, so $\{ r_{\lambda }  \}_{\lambda \in \mathcal{P}_n }$ is linearly independent.
\end{proof}

\begin{rem}
We have obtained in the proof of \cref{thm:graphcomukernel} that $\{ r_{\lambda}\}_{\lambda \vdash n}$ is a basis for $Sym_n$.
This basis is different from other ``chromatic bases'' proposed in \cite{willigenburg15}.
The proof gives us a recursive way to compute the coefficients $\zeta_{\lambda} $ on the span $\Psi_{\mathbf{G}} ( G ) = \sum_{\lambda } \zeta_{\lambda} r_{\lambda }$.
It is then natural to ask if combinatorial properties can be obtained for these coefficients, which are isomorphism invariants.


Similarly in the non-commutative case, we obtain that $\mathbf{WSym}_n $ is spanned by $\{\boldsymbol{\Upsilon }_{\mathbf{G}} (K_{\parpi }^c) \}_{\parpi\vdash [n] }$, and so other coefficients arise.
We can again ask for combinatorial properties of these coefficients.
\end{rem}

\subsection{The augmented chromatic invariant}

Consider the ring of power series $\mathbb{K}[[x_1, x_2, \dots ;q_1, q_2, \dots ]] $ on two countably infinite collections of commuting variables, and let $R$ be such ring modulo the relations $ q_i(q_i - 1)^2 = 0$.

Consider the graph invariant
$\tilde{\Psi } (G) = \sum_f x_f \prod_i q_i^{c_G(f, i)}$ in $R$,
where the sum runs over \textbf{all} colorings $f$ of $G$, and $c_G(f, i)$ stands for the number of monochromatic edges of color $i$ in the coloring $f$ (i.e. edges $\{v_1, v_2\}$ such that $f(v_1 ) = f(v_2) = i$).

For instance, if $G = K_2$, then $\tilde{\Psi}(G) = 2\sum_{1 \leq i < j} x_i x_j + \sum_{1 \leq i} x_i^2 q_i$.
If we consider $G= K_3$ then we have 
$$\tilde{\Psi} (G)= 6\sum_{1 \leq i < j < k}x_i x_j x_k  + 3 \sum_{i \neq j} x_i x_j^2 q_j + \sum_{1 \leq i } x_i^3 q_i^3\, .$$
Note that we can simplify further with the relation $q_i^3 = 2 q_i^2 - q_i $.

A main property of this graph invariant is that it can be specialized to the chromatic symmetric function, by evaluating each variable $q_i$ to zero.
Another property of this graph invariant is the following:
\begin{prop}\label{prop:goodinv}
We have that $\ker \tilde{\Psi} = \ker \Psi_{\mathbf{G}}$.
In particular, for graphs $G_1, G_2$ we have $\Psi_{\mathbf{G}} (G_1) = \Psi_{\mathbf{G}}(G_2 )$ if and only if $\tilde{\Psi}(G_1) = \tilde{\Psi}(G_2)$.
\end{prop}

Take, for instance, the celebrated tree conjecture introduced in \cite{stanley95}:

\begin{conj}[Tree conjecture on chromatic symmetric functions]\label{conj:tree}
If two trees $T_1, T_2$ are not isomorphic, then $\Psi_{\mathbf{G}}(T_1 ) \neq \Psi_{\mathbf{G}}(T_2)$.
\end{conj}

Consequently, from \cref{prop:goodinv}, the tree conjecture is equivalent to the following conjecture:

\begin{conj}\label{conj:tildatree}
If two trees $T_1, T_2$ are not isomorphic, then $\tilde{\Psi}(T_1 ) \neq \tilde{\Psi}(T_2)$.
\end{conj}

One strategy that has been employed to show that a family of non-isomorphic trees is distinguished by their chromatic symmetric function is to construct said trees using its coefficients over several bases, see for instance \cite{orellana14}, \cite{smith15} and \cite{aliste14}.
The graph invariant $\tilde{\Psi}$ provides more coefficients to reconstruct a tree, because $\Psi $ results from $\tilde{\Psi} $ after the specialization $q_i = 0 $.
So, employing the same strategy to prove \cref{conj:tildatree} is \textit{a priori} easier than to approach \cref{conj:tree} directly.

This shows us that the kernel method can also give us some light on other graph invariants: they may look stronger than $\Psi $, but are in fact as strong as $\Psi $ if they satisfy the modular relations.

\begin{proof}[Proof of \cref{prop:goodinv}]
Note that we have $\tilde{\Psi} (G) |_{q_i =0 \, \, i=1, 2, \dots } = \Psi_{\mathbf{G}}(G)$.
This readily yields $\ker \tilde{\Psi } \subseteq \ker \Psi_{\mathbf{G} }$.
To show that $\ker \tilde{\Psi } \supseteq \ker \Psi_{\mathbf{G} }$, we need only to show that the modular relations and the isomorphism relations belong to $\ker \tilde{\Psi } $.
For the isomorphism relations, this is trivial.

Let $l = G - G\cup \{e_1 \} - G\cup \{e_2 \} + G \cup \{e_1, e_2 \}$ be a generic modular relation on graphs, i.e. $\{ e_1, e_2, e_3 \} $ are edges that form a triangle between the vertices $\{v_1, v_2, v_3 \}$, with $e_3 \in G, e_1, e_2 \not\in G$.
Say that $e_1 = \{v_2, v_3 \}$, $e_2 = \{v_3, v_1 \}$ and $e_3 = \{v_1, v_2 \}$.
The proposition is proved if we show that $\tilde{\Psi}( l  )= 0 $.

For a coloring $f$ of a graph $H$ and a monochromatic edge $e$ in $H$, define $c(e)$ the color of the vertices of $e$.
Abbreviate $\mathbb{1}[e \text { is monochromatic} ]= m(e)$.
With this, we use the abuse of notation $ q_{c(e)}^{m(e) }$ even when $e$ is not monochromatic, in which case we have $ q_{c(e)}^{m(e) } = q_{c(e)}^0 = 1$.
Then
\begin{equation}\label{eq:edgeresp}
\prod_i q_i^{c_H(f, i )} =   \prod_{e\text{ monochromatic } }   q_{c(e)} = \prod_{e\in E(H) } q_{c(e)}^{m(e) }\, . 
\end{equation}

Set
\begin{equation}\label{eq:sf}
\begin{split}
 s_f :&= \prod_iq_i^{c_G(f, i)} - \prod_iq_i^{c_{G\cup \{e_1 \}}(f, i)} - \prod_iq_i^{c_{G\cup \{e_2 \}}(f, i)} + \prod_i q_i^{c_{G \cup \{e_1, e_2 \}}(f, i)}\\
				 &= \left( 1 - q_{c(e_1)}^{m(e_1) } - q_{c(e_2)}^{m(e_2) } + q_{c(e_1)}^{m(e_1) }q_{c(e_2)}^{m(e_2) } \right) \prod_{e\in E(G) } q_{c(e)}^{m(e) } \\
				 &= \left( 1 - q_{c(e_1)}^{m(e_1) }  \right) \left( 1- q_{c(e_2)}^{m(e_2) } \right) \prod_{e\in E(G) } q_{c(e)}^{m(e) }\, ,
\end{split}
\end{equation}
and observe that $\tilde{\Psi }(l ) = \sum_f x_f s_f $.
Fix a coloring $f$.
We now show that $s_f$ is always zero.

It is easy to see that if either $e_1$ or $e_2$ are not monochromatic, then $q_{c(e_1)}^{m(e_1) } = 1$, respectively $q_{c(e_2)}^{m(e_2) }=1$, which implies that $s_f=0$.

It remains then to consider the case where $\{v_1, v_2, v_3  \} $ is monochromatic.
Without loss of generality, say it is of color $a$.

Then, $q_{c(e_1)}^{m(e_1) } = q_a = q_{c(e_2)}^{m(e_2) } $.
Further, we have that $q_{c(e_3)}^{m(e_3) }  = q_a$, so in $R$ we have
$$s_f =	(1 - q_a)^2 q_a  \prod_{e\in E(G)\setminus \{e_3\} } q_{c(e)}^{m(e) } = 0 \, . $$

So $\tilde{\Psi }(l) = \sum_f x_f s_f = 0$.

In conclusion, any modular relation and any isomorphism relation is in $\ker \tilde{\Psi }$.
From \cref{thm:graphcomukernel} we have that $ \ker \Psi_{\mathbf{G}} \subseteq \ker \tilde{\Psi} $, so we conclude the proof.
\end{proof}

\medskip

It is clear that \cref{prop:goodinv} was established in an indirect way, by studying the kernel of the maps $\tilde{\Psi } $ and $\Psi$, instead of relating the coefficients of both invariants in some basis.

In \cref{app:coef} we relate the coefficients of both invariants in \cref{cor:weakgoodiv}.
Our original goal of establishing \cref{prop:goodinv} without using \cref{thm:graphcomukernel} directly is not accomplished, which lends more strength to this kernel method.



\section{The CSF on hypergraphic polytopes\label{ch4}}

\subsection{Poset structures on compositions}

In this chapter we consider generalized permutahedra and hypergraphic polytopes in $\mathbb{R}^n$.
Recall that with a set composition $\opi =S_1 | \dots | S_k \in \oPi_n $ we have the associated total preorder $R_{\opi}$, and for a non-empty set $A \subseteq [n] $, we define the set $A_{\opi} = A\cap S_i$ where $i$ is as small as possible so that $A_{\opi } \neq \emptyset $.
We refer to $A_{\opi}$ as the minima of $A$ in $R_{\opi }$.

Finally, recall as well that, for a hypergraphic polytope $\mathfrak{q}$, $\mathcal{F}(\mathfrak{q}) \subseteq 2^I\setminus \{ \emptyset \}$ denotes the family of sets $J\subseteq I$ such that the coefficients in \eqref{eq:gpcoefs} satisfy $a_J > 0$.
For a set $A \subseteq 2^I\setminus \{ \emptyset \} $, we write $\mathcal{F}^{-1}(A)$ for the hypergraphic polytope $\mathfrak{q} = \sum_{J\in A} \mathfrak{s}_J $.
We write $a = \pt $ whenever $a$ is a point polytope.

For a generalized permutahedron $\mathfrak{q}$ and a real coloring $f$ with set composition type $\opi $, recall that we write $\mathfrak{q}_f = \mathfrak{q}_{\opi }$ for the corresponding face, without ambiguity, see \cref{lm:normfan}.

\begin{defin}[Basic hypergraphic polytopes and a preorder in set compositions]

For $\opi \in \oPi_n$, the corresponding \textit{basic hypergraphic polytope} is the fundamental hypergraphic polytope $\mathfrak{p}^{\opi} = \mathcal{F}^{-1} \{ A  |\#  A_{\opi} = 1 \} $.

Consider two set composition $\opi_1, \opi_2 \in \oPi_n$.
If for any non-empty $A \subseteq [n]$ we have $\# A_{\opi_1} = 1  \Rightarrow \#  A_{\opi_2} = 1$, we write $\opi_1 \preceq \opi_2 $.
Equivalently, $\opi_1 \preceq \opi_2 $ if $\mathcal{F} (\mathfrak{p}^{\opi_1}) \subseteq \mathcal{F}(\mathfrak{p}^{\opi_2})$.
With this, $\preceq $ is a preorder, called \textit{singleton commuting preorder} or \textit{SC preorder}.
This nomenclature is motivated by \cref{prop:combinterpretation}.

Additionally, we define the equivalence relation $\sim $ in $\oPi_n $ as $\opi \sim \otau $ whenever $\# A_{\opi } = 1 \Leftrightarrow \#  A_{\otau} = 1 $ for all non-empty sets $A \subseteq [n] $. 
Note that $\opi \sim \otau $ if and only if  $\mathfrak{p}^{\opi } = \mathfrak{p}^{\otau }$.
We write $[\opi ]$ for the equivalence class of $\opi $ under $\sim $, and write $\mathfrak{p}^{[\opi]} = \mathfrak{p}^{\opi } $, without ambiguity.

The preorder $\preceq $ projects to an order in $\oPi_n /_{\sim }$.
In \cref{prop:growthdimim}, we find an asymptotic formula for the number of equivalence classes of $\sim $.
\end{defin}

\begin{smpl}
We see here the preorder $\preceq $ for $n=3$, and the corresponding order in the equivalence classes of $\sim $.
The set compositions $\opi$ such that $\lambda (\opi ) = (1, 1, 1) $, are in bijection with permutations on $\{1, 2, 3\}$, we call these the \textit{permutations in} $\oPi_3$.
For a permutation $\opi$ in $\oPi_3$ and a non-empty subset $A \subseteq \{1, 2, 3\}$, we have $\#  A_{\opi} = 1 $.
Hence, permutations are maximal elements in the singleton commuting preorder, and form an equivalence class of $\sim $.
This is called the trivial equivalence class.

We also observe that if $A$ is such that $\# A_{23|1}  = 1 $, then $\{2, 3\} \not\subseteq A$ and so we have that $\# A_{1|23} = 1$.
It follows that $1|23 \succ 23|1 $.
The remaining structure of the preorder in $\oPi_3$ is in \cref{fig:oPiSCpre}, where we collapse equivalence classes into vertices and draw the corresponding poset in its Hasse diagram.

\begin{figure}[ht]\centering
\includegraphics[scale=0.7]{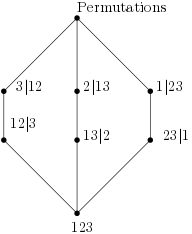}
\caption{\label{fig:oPiSCpre}The SC order in $\oPi_3/_{\sim}$.}
\end{figure}

\end{smpl}

For $n=4 $ things are more interesting, as we have non-trivial equivalence classes.
For instance, we have $[12|3|4]=\{12|4|3\} $.

\begin{prop}\label{prop:combinterpretation}
Let $\opi, \otau\in \oPi_n$.
Then we have that $\opi \preceq \otau \Rightarrow \makepar ( \opi ) \geq \makepar (  \otau )$.

Additionally, $\opi \sim \otau $ if and only if all the following happens: 
\begin{enumerate}
\item We have $\makepar ( \opi )  =\makepar (\otau ) $, and;

\item For each pair $(a, b)$ with $a, b \in [n]$ that satisfies both $a\, R_{\opi}\, b $ and $b\, R_{\otau}\, a $, either $\{a\}, \{b\}\in \makepar ( \opi )$ or $a\sim_{\makepar ( \opi )} b $.
\end{enumerate}

In particular, $\alpha (\opi ) = \alpha (\otau )$.
\end{prop}

Property 2. will be called the SC property.
The equivalence classes of $\sim $ have a clear combinatorial description via \cref{prop:combinterpretation}.
In particular, we see that $\opi_1 \sim \opi_2$ if all blocks are the same and in the same order, with possible exceptions between singletons.
For instance, we have that $12|3|4|5|67 \sim 12|3|5|4|67$ but $12|3|4|5|67 \not\sim 3|12|4|5|67$.

We are also told in \cref{prop:combinterpretation} that the map $\makepar: \oPi_n \to \mathbf{P}_n $ flips the SC preorder with respect to the coarsening order $\leq $ in $\mathbf{P}_n$.

\begin{proof}[Proof of \cref{prop:combinterpretation}]
Write $\parpi, \partau $ for the underlying set partitions $\makepar (\opi ) , \makepar (\otau )$, respectively.
Suppose that $\opi \preceq \otau$ and take $i, j$ elements of $[n]$ such that $i \sim_{\partau } j $.
Then $\{ i, j\} \not \in  \mathcal{F} ( \mathfrak{p}^{\otau} ) \supseteq \mathcal{F} ( \mathfrak{p}^{\opi} )$.
This implies that $\# \{i, j\}_{\opi } \neq 1 $, hence $ i \sim_{\parpi } j $.
Since $i, j$ are generic, we have that $\parpi  \geq \partau $.
This concludes the first part.

For the second part, we will first show the direct implication.
Suppose that $ \mathfrak{p}^{\opi} = \mathfrak{p}^{\otau} $.
It follows from above that $\parpi = \partau$.
Our goal is to establish the SC property.

Take $a, b$ that are in distinct blocks in $\parpi $, such that both $a\,  R_{\opi }\, b $ and $b\, R_{\otau }\, a $.
For sake of contradiction let $c \neq a$ be such that $c\sim_{\parpi } a $.
Then $\{a, b, c \}_{\opi} = \{a, c \} $, which is not a singleton.
However, we have that $\{a, b, c \}_{\otau} = \{ b \}$ is a singleton, which is a contradiction with $\opi \sim \otau $.
This contradicts the assumption that $a \sim_{\parpi } c$, so we conclude that $\{ a \} \in \parpi $.
Similarly we obtain that $\{ b \} \in \parpi $.
This shows the SC property.

For the reverse implication, suppose that $\opi, \otau $ are such that $\parpi =  \partau $ and satisfy the SC property.
Our goal is to show that $\opi \sim \otau $.
For sake of contradiction, take some nonempty set $A \subseteq [n] $ such that $A_{\opi } = \{ a \} $ is a singleton, but $\#  A_{\otau } \neq 1$.
Finally, take an element $b \in A_{\otau }$, so that $a, b \in A$.
We immediately have $a \, R_{\opi }\, b$, $b\, R_{\otau }\, a$.

Since $\parpi = \partau $, either $A_{\opi } = A_{\otau } $ or $A_{\opi }\cap A_{\otau } =\emptyset$.
Since $A_{\opi } \neq A_{\otau } $, they are disjoint and in particular $a \not \sim_{\parpi } b$.
By the SC property we conclude that both $ \{ a \}, \{ b \} \in \parpi $, contradicting that $\# A_{\otau }\neq 1 $.
That $\# A_{\otau } = 1 \Rightarrow \# A_{\opi } = 1$ follows similarly, concluding the proof.

Finally, whenever $\opi \sim \otau $, the SC property gives us $\alpha(\opi ) = \alpha (\otau )$.
\end{proof}

The following definition focus on the algebraic counterpart of $\sim $.

\begin{defin}\label{def:SCspace}
Consider the quasisymmetric functions $\mathbf{N}_{[\opi ] } = \sum_{\otau \in [\opi] } \mathbf{M}_{\otau} $, which are linearly independent.
The \textit{singleton commuting space}, or \textbf{SC} for short, is the graded vector subspace of \textbf{WQSym} spanned by $\biguplus_{n\geq 0} \{ \mathbf{N}_{[\opi ] } \, : \, [\opi ] \in \oPi_n / \sim \} $ .

\end{defin}

In \cref{lm:liness}, we show that \textbf{SC} is the image of $\uhsm_{\mathbf{HGP}}$.
As a consequence, \textbf{SC} is a Hopf algebra.

We turn to some properties of hypergraphic polytopes in the next lemma:

\begin{lm}[Vertices of a hypergraphic polytope]\label{lm:vcc}
Let $\mathfrak{q}$ be a hypergraphic polytope and let $f$ be a real coloring.

Then, we have that $\mathfrak{q}_f = \pt$ if and only if $\#  A_{\opi(f)} = 1 $ for each $ A \in \mathcal{F} (q )$.
In particular, if $\opi(f_1) \sim \opi(f_2)$ then $\mathfrak{q}_{f_1} = \pt \Leftrightarrow \mathfrak{q}_{f_2} = \pt$.
\end{lm}

\begin{proof}
Write $\mathfrak{q} = \sum_{A \in \mathcal{F}(\mathfrak{q} ) } a_A \mathfrak{s}_{  A } $, for coefficients $a_A\geq 0 $.
Computing the face corresponding to $f$ on both sides, we obtain that $\mathfrak{q}_f = \pt $ if and only if 
$$ \sum_{A \in \mathcal{F}(\mathfrak{q} )  } a_A (\mathfrak{s}_A )_f = \pt \, ,$$
or equivalently, if $(\mathfrak{s}_A )_f = \pt $ for each $A \in \mathcal{F}(\mathfrak{q} ) $.

We observed in \cref{eq:facesimpl} that $(\mathfrak{s}_A)_f =\mathfrak{s}_{ A_{\opi ( f)} }$.
Hence, we conclude that $\mathfrak{q}_f = \pt $ if and only if $\#  A_{\opi ( f)} = 1 $ for each $A \in \mathcal{F}(\mathfrak{q} ) $, as desired.

To show the last part of the lemma, just observe that $\#  A_{\opi} = 1 $ only depends on the equivalence class of $\opi$.
\end{proof}

The following corollary is immediate from \eqref{eq:uhsmiswqsym} and \cref{lm:vcc}.

\begin{cor}\label{cor:imSC}
The image of $\uhsm_{\mathbf{HGP}}$ is contained in the \textbf{SC} space, i.e. for any hypergraphic polytope $\mathfrak{q}$ we have that
$$ \boldsymbol{\Upsilon }_{\mathbf{HGP} }(\mathfrak{q} ) \in \mathbf{SC}\, . $$
\end{cor}

Another consequence of \cref{lm:vcc} is that we have $\mathfrak{p}^{\opi}_{\otau} = \pt $ precisely when $ \# A_{\opi } = 1 \Rightarrow \# A_{\otau} = 1 $, i.e. when $\opi \preceq \otau $.
It follow from \eqref{eq:uhsmiswqsym} that:

\begin{equation}\label{eq:fundnestocase}
\boldsymbol{\Upsilon }_{\gpHa } (\mathfrak{p}^{\opi } ) = \sum_{\opi \preceq \otau} \mathbf{M}_{\otau } \, .
\end{equation}

As presented, \eqref{eq:fundnestocase} seems to shows that the transition matrix of $\{ \boldsymbol{\Upsilon }_{\gpHa } (\mathfrak{p}^{\opi } ) \}_{\opi \in \oPi_n } $ over the monomial basis is upper triangular.
Since $\preceq $ is not an order but a preorder, that is not the case.
The related result that we can establish is the following:

\begin{lm}\label{lm:liness}
\label{lm:spanning}
The family $\{ \boldsymbol{\Upsilon }_{\mathbf{HGP }} (\mathfrak{p}^{[\opi]}) \}_{ [\opi] \in \oPi_n/ \sim }$ forms a basis of \textbf{SC}.
In particular, we have $ \im \boldsymbol{\Upsilon }_{\mathbf{HGP} } =  \mathbf{SC}$.
\end{lm}

\begin{proof}
From \eqref{eq:fundnestocase} we have the following triangularity relation:
\begin{equation}\label{eq:blocktrignesto}
\mathbf{\Psi } (\mathfrak{p}^{[\opi ]} ) = \sum_{\opi \sim \otau }M_{\otau} + \sum_{\substack{\opi \prec \otau \\ \opi \not\sim \otau }}M_{\otau } = N_{[\opi ]} + \sum_{[\opi ] \prec [\otau ]} N_{[\otau] } \, ,
\end{equation}
where we take the projecton of the preorder $\preceq $ into the corresponding order in $\oPi_n / \sim $.

Thus, $\{ \boldsymbol{\Upsilon }_{\mathbf{HGP }} (\mathfrak{p}^{[\opi]}) \}_{ [\opi] \in \oPi_n/ \sim }$ is another basis of \textbf{SC}.
From \cref{cor:imSC}, we conclude that $\im \boldsymbol{\Upsilon }_{\mathbf{HGP}} = \mathbf{SC}$.
\end{proof}

In the commutative case, we wish to carry the triangularity of the monomial transition matrix in \eqref{eq:blocktrignesto} into a new smaller basis in $QSym$.

For that, we project the order $\preceq $ into an order $\leq '$ in $\mathcal{C}_n$ as follows: we say that $\alpha \leq' \beta $ if we can find set compositions $\opi, \otau $ that satisfy $\opi \preceq \otau $, $\alpha (\opi ) = \alpha$ and $\alpha (\otau ) = \beta$.
We will see that this projection is akin to the projection of the coarsening order of set partitions to the coarsening order on partitions.
In particular, it preserves the desired upper triangularity.

\begin{lm}\label{lm:ordercomps}
The relation $\leq' $ on $\mathcal{C}_n $ is an order and satisfies $\opi \preceq \otau \Rightarrow \alpha ( \opi ) \leq' \alpha (\otau ) $.
\end{lm}

Recall that permutations of $[n]$ act on set compositions: if $\opi = A_1| \dots |A_k \in \oPi_n$, and $\phi\in S_n $, then $\phi(\opi ) \in \oPi_n$ is the set composition $\phi(S_1) | \dots | \phi(S_k)$.

\begin{proof}[Proof of \cref{lm:ordercomps}]
We only need to check that $\leq'$ as defined is indeed an order, as it is straightforward that $\opi \preceq \otau \Rightarrow \alpha (\opi ) \leq' \alpha ( \otau )$.

Reflexivity of $\leq' $ trivially follows from the definition of $\preceq $.
To show antisymmetry of $\leq '$, it is enough to establish that if $\opi_1 \preceq \otau_2$, $\otau_1 \preceq \opi_2 $ are set compositions such that $\alpha := \alpha (\opi_1) = \alpha (\opi_2)$ and $\beta := \alpha (\otau_1) = \alpha (\otau_2)$, then $\alpha = \beta$.

Indeed, if $\alpha (\opi_1) = \alpha (\opi_2)$ then there is a permutation $\phi$ in $[n]$ that satisfies $\phi ( \opi_1 ) = \opi_2 $.
Then, $\phi$ lifts to a bijection between $ \mathcal{F} (\mathfrak{p}^{\opi_1 } ) $ and $ \mathcal{F} (\mathfrak{p}^{\opi_2 } )$; in particular, they have the same cardinality.
Similarly, $ \mathcal{F} (\mathfrak{p}^{\otau_1 } ) $ and $ \mathcal{F} (\mathfrak{p}^{\otau_2 } )$ have the same cardinality.

But since $\opi_1 \preceq \otau_2$, $\otau_1 \preceq \opi_2 $, i.e. $ \mathcal{F} (\mathfrak{p}^{\opi_1 } ) \subseteq \mathcal{F} (\mathfrak{p}^{\otau_2 } )$ and $\mathcal{F} (\mathfrak{p}^{\otau_1 } ) \subseteq \mathcal{F} (\mathfrak{p}^{\opi_2 } )$, it follows that $\mathcal{F} (\mathfrak{p}^{\opi_1 } ) = \mathcal{F} (\mathfrak{p}^{\otau_1 } )$, and so $\opi_1 \sim \otau_1$.
From \cref{prop:combinterpretation}, we have that $\alpha = \alpha( \opi_1 ) = \alpha(\otau_1 ) = \beta$, and the antisymmetry follows.

To show transitivity, take compositions such that $\alpha \leq' \beta$ and $\beta  \leq' \sigma $, i.e. there are set compositions $\opi \preceq \otau_1$ and $\otau_2 \preceq \ogamma$ such that $\alpha (\opi ) = \alpha$, $\alpha (\otau_1 ) = \alpha (\otau_2) = \beta$ and $\alpha(\ogamma ) = \sigma$.
Take a permutation $\phi $ in $[n]$ such that $\phi(\otau_1 ) = \otau_2$ and call $\odelta = \phi(\opi )$, note that $\alpha (\odelta ) = \alpha (\opi )= \alpha $.

We claim that $\odelta \preceq \otau_2$. It follows that $\odelta \preceq \ogamma$ and $\alpha \leq' \sigma $, so the transitivity of $\leq'$ also follows.
Take $A \subseteq [n]$ nonempty such that $\#  A_{\odelta} =  1$.
Then $\#  \phi^{-1}(A)_{\opi} = 1$ and from $\opi \leq \otau_1 $ it follows that $\# \phi^{-1}(A)_{\otau_1} = 1$.
From $\otau_2 = \phi(\otau_1) $ we have that $\#  A_{\otau_2 } = 1$.
Since $A$ is generic such that $\#  A_{\odelta} = 1 $, we conclude that $\odelta \preceq \otau_2$, as envisaged.
\end{proof}

\subsection{The kernel and image problem on hypergraphic polytopes}

Recall that a fundamental hypergraphic polytope in $\mathbb{R}^{[n]}$ is a polytope of the form $\sum_{\emptyset \neq J \subseteq [n] } a_J\mathfrak{s}_J$ where each $a_J \in \{0, 1\}$.
In particular, a fundamental hypergraphic polytope can be written as $\mathfrak{q} =\mathcal{F}^{-1} (\mathcal{A})$ for some family of non-empty subsets of $[n]$.

In the following proposition, we reduce the problem of describing the kernel of $\uhsm_{\mathbf{HGP}}$ to the subspace of $\mathbf{HGP}$ spanned by the fundamental hypergraphic polytopes.

\begin{prop}[Simple relations for $\boldsymbol{\Upsilon }_{\mathbf{HGP}} $]\label{prop:nestosmplmodrel}
If $\mathfrak{q}_1, \mathfrak{q}_2$ are two hypergraphic polytopes such that $\mathcal{F}(\mathfrak{q}_1 ) = \mathcal{F}(\mathfrak{q}_2) $, then 
$$\uhsm_{\mathbf{HGP}}(\mathfrak{q}_1) = \uhsm_{\mathbf{HGP}}(\mathfrak{q}_2 )\, .$$
\end{prop}

It remains to discuss the kernel of the map $\uhsm_{\mathbf{HGP}} $ in the space of fundamental hypergraphic polytopes $\{ \mathcal{F}^{-1}(\mathcal{A}) | \, \, \mathcal{A} \subseteq 2^{[n]}\setminus \{\emptyset \} \}$.
For non-empty sets $A \subseteq [n]$, define $\text{Orth } A = \{\opi \in \oPi_n | \# A_{\opi } = 1 \}$.
We now exhibit some linear relations of the chromatic function on fundamental hypergraphic polytopes.

\begin{thm}[Modular relations for $\boldsymbol{\Upsilon }_{\mathbf{HGP}} $]\label{thm:nestomodrel}
Let $\mathcal{A}, \mathcal{B} $ be two disjoint families of non-empty subsets of $[n]$.
Consider the hypergraphic polytope $\mathfrak{q} =  \mathcal{F}^{-1}( \mathcal{A})$, and take $ \mathcal{K} = \cup_{A \in \mathcal{A}} (\Orth  A )^c $,  and $ \mathcal{J} = \cup_{ B \in \mathcal{B} } \Orth \,  B$, families of set compositions.

Suppose that $\mathcal{K}\cup \mathcal{J} = \oPi_n $.
Then,
$$\sum\nolimits_{\mathcal{T} \subseteq \mathcal{B}}\, (-1)^{\# \mathcal{T}} \,\boldsymbol{\Upsilon }_{\mathbf{HGP}} \left[ \mathfrak{q} + \mathcal{F}^{-1}  ( \mathcal{T} ) \right] = 0\, ,$$
where the sum $\mathfrak{q} + \mathcal{F}^{-1}  ( \mathcal{T} ) $ is taken as the Minkowski sum.
\end{thm}

\begin{figure}[ht]
\centering
\includegraphics[scale=0.50]{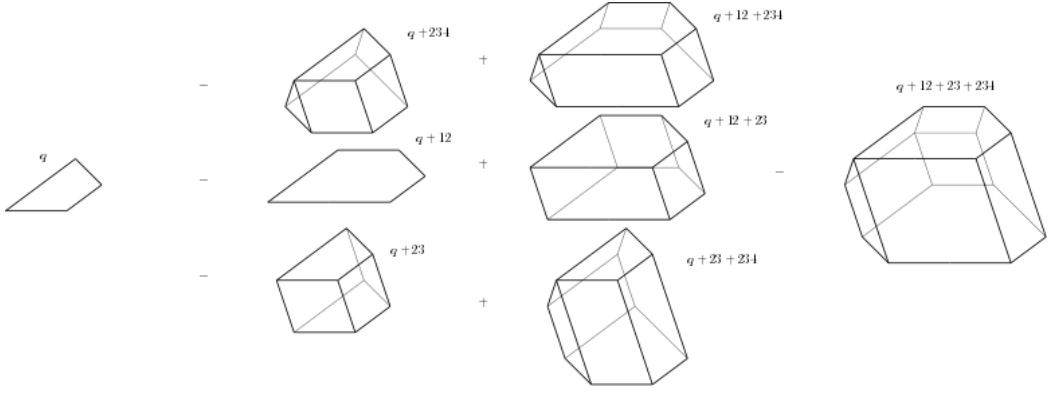}
\caption{\label{fig:modrelnesto}A modular relation on hypergraphic polytopes for $\mathfrak{q} = \mathcal{F}^{-1} (  \{1, 4\}, \{1, 2, 4\} ) $.}
\end{figure}

The sum $\sum_{\mathcal{T} \subseteq \mathcal{B}}\, (-1)^{\# \mathcal{T}} \, \left[ \mathfrak{q} + \mathcal{F}^{-1}  ( \mathcal{T} ) \right]   $ is called a \textit{modular relation on hypergraphic polytopes}.
An example can be observed in \cref{fig:modrelnesto} for $n = 4$, where we take the families $\mathcal{A} = \{ \{1, 4 \}, \{1, 2, 4 \}  \}$,  $\mathcal{B} =  \{\{1, 2\}, \{2, 3\}, \{2, 3, 4\} \} $.

\begin{proof}[Proof of \cref{thm:nestomodrel}]
Write $\eta_f ( \mathfrak{q} ) = \mathbb{1} [ f \text{ is } \mathfrak{q}-\text{generic } ]$. 
The expansion of $\boldsymbol{\Upsilon }_{\mathbf{HGP}}$ for general hypergraphic polytopes $\mathfrak{q}$ is given by 
$$ \mathbf{\Psi }_{\mathbf{HGP}} (\mathfrak{q}) = \sum_{f } x_f \eta_f(\mathfrak{q})  =  \sum_{f \text{ is }\mathfrak{q}\text{-generic}} x_f \, .$$

For short, write $\mathrm{MRN} $ for the modular relation for hypergraphic polytopes at hand.
Hence:

\begin{equation}\label{eq:modrelnesto1}
\begin{split}
\mathbf{\Psi }_{\mathbf{HGP}}(\mathrm{MRN}) =& \sum_{\mathcal{T} \subseteq \mathcal{B}} (-1)^{\# \mathcal{T}} \mathbf{\Psi }_{\mathbf{HGP}} \left[ \mathfrak{q} + \sum_{T \in \mathcal{T}} \mathfrak{s}_{T}  \right]\\
=&  \sum_{\mathcal{T} \subseteq \mathcal{B}} (-1)^{\# \mathcal{T}}\sum_f x_f \eta_f\left( \mathfrak{q} + \sum_{T \in \mathcal{T}} \mathfrak{s}_{T} \right) \\
=& \sum_f x_f \left[\sum_{\mathcal{T} \subseteq \mathcal{B}} (-1)^{\# \mathcal{T}} \eta_f \left( \mathfrak{q} + \sum_{T \in \mathcal{T}} \mathfrak{s}_{T} \right) \right]\, .
\end{split}
\end{equation}

We note from \cref{lm:minkfaces,lm:vcc} that for hypergraphic polytopes $\mathfrak{q}, \mathfrak{p}$, any coloring $f$ that is not $\mathfrak{q}$-generic is not $(\mathfrak{q} + \mathfrak{p})$-generic.
Hence, if $ \eta_f (\mathfrak{q}) =0 $ then it follows that $ \eta_f \left( \mathfrak{q} + \sum_{T \in \mathcal{T}} \mathfrak{s}_{T} \right) ~=~0$ for any $\mathcal{T}\subseteq \mathcal{B}$.
We restrict the sum to $\mathfrak{q}$-generic colorings.

Further, define $\mathcal{B}(f) = \{ B\in \mathcal{B} | f \text{ is } \mathfrak{s}_{B}-\text{ generic } \} $.
Again according to \cref{lm:minkfaces,lm:vcc}, we have that $ \eta_f \left( \mathfrak{q} + \sum_{T \in \mathcal{T}} \mathfrak{s}_{T} \right) ~=~1$ exactly when $\mathcal{T} \subseteq \mathcal{B}(f) $, 
so the equation \eqref{eq:modrelnesto1} becomes

\begin{equation}
\begin{split}\label{eq:modrelnesto2}
\mathbf{\Psi }_{\mathbf{HGP}}(MRN)
=& \sum_f x_f \left[\sum_{\mathcal{T} \subseteq \mathcal{B}} (-1)^{\# \mathcal{T}} \eta_f \left( \mathfrak{q} + \sum_{T \in \mathcal{T}} \mathfrak{s}_{T} \right) \right] \\
=& \sum_f x_f \left[\sum_{\mathcal{T} \subseteq \mathcal{B}(f)} (-1)^{\# \mathcal{T}} \right] = \sum_{\substack{f \, \, \mathfrak{q} \text{- generic} \\ \mathcal{B}(f) =\emptyset }} x_f \, .
\end{split}
\end{equation}

It suffices to show that no coloring $f$ is both $\mathfrak{q}$-generic and satisfies $B(f) =\emptyset $.
Suppose otherwise, and take such $f$.
If $f$ is  $\mathfrak{q}$-generic, we have $\opi(f) \not\in \mathcal{K} $.
If $B(f) = \emptyset$ we have that $\opi(f) \not \in \mathrm{ Orth } \, B_j $ for any $j$, hence $\opi(f) \not \in \mathcal{J}$. 
Given that $f\in \mathcal{K}\cup \mathcal{J} =\oPi_n $, we have a contradiction.
It follows that \cref{eq:modrelnesto2} becomes $\mathbf{\Psi }_{\mathbf{HGP}}(\mathrm{MRN} ) = 0 $,
which concludes the proof.
\end{proof}

\begin{rem}
Recall that $Z: \mathbf{G} \to \mathbf{GP}$ is the graph zonotope map.
It can be noted that, if $l = G - G \cup \{e_1\} - G \cup \{e_2\} + G \cup \{e_1, e_2\}$ is a modular relation on graphs, then $Z(l) $ is the modular relation on hypergraphic polytopes corresponding to $\mathfrak{q} = Z(G) $ ( i.e. $\mathcal{A} = E(G) $), and $\mathcal{B} = \{e_1, e_2 \}$ ) .
In this case, the condition $\mathcal{K} \cup \mathcal{J}  = \oPi_n $ follows from the fact that no proper coloring of $G$ is monochromatic in both $e_1$ and $e_2$.
\end{rem}

Recall that we set $\mathfrak{p}^{\opi} = \mathcal{F}^{-1} (\{A \subseteq [n] | \# A_{\opi} = 1 \} ) $.
This only depends on the equivalence class $[\opi ] $ under $\sim$, and we may write the same polytope as $\mathfrak{p}^{[\opi ] }$.
These are called basic hypergraphic polytopes and are a particular case of fundamental hypergraphic polytopes.

To prove \cref{thm:nestokernel}, we follow roughly the same idea as in the proof of \cref{thm:graphkernel}:
We use the family of hypergraphic polytopes $\{ \mathfrak{p}^{[\opi]} \}_{ [\opi ] \in \oPi_n / \sim  }$ to apply \cref{lm:kernandimdesc}, whose image by $\boldsymbol{\Upsilon }_{\gpHa}$ is linearly independent.
Recall that in \cref{lm:liness}, we established that it spans the image of $\uhsm_{\mathbf{HGP}}$.

\begin{proof}[Proof of \cref{thm:nestokernel}]
First recall that $\mathbf{HGP}_n$ is a linear space generated by the hypergraphic polytopes in $\mathbb{R}^n$.
According to \cref{prop:nestosmplmodrel}, to compute the kernel of $\uhsm_{\mathbf{HGP}}$, it suffices us study the span of the fundamental hypergraphic polytopes.
Fix a total order $\tilde{\geq } $ on fundamental hypergraphic polytopes $\mathfrak{q}$ so that $\# \mathcal{F}(\mathfrak{q}) $ is non decreasing.

We apply \cref{lm:kernandimdesc} with \cref{thm:nestomodrel} to this finite dimensional subspace of $\mathbf{HGP}_n$.

\cref{lm:liness} guarantees that $\{\boldsymbol{\Upsilon }_{\gpHa } ( \mathfrak{p}^{ [\opi]} )\}_{ [\opi]  \in \oPi_n/ \sim }$ is linearly independent.
Therefore, it suffices to show that for any fundamental hypergraphic polytopes $\mathfrak{q}$ that is not a basic hypergraphic polytope, we can write some modular relation $b$ as $b =  \mathfrak{q} + \sum_i \lambda_i \mathfrak{q}_i$, 
where $\# \mathcal{F}(\mathfrak{q}) < \# \mathcal{F}(\mathfrak{q}_i ) \,\forall i$.
Indeed, it would follow from \cref{lm:kernandimdesc} that the simple relations and the modular relations on hypergraphic polytopes span $\ker \boldsymbol{\Upsilon }_{\mathbf{HGP}}$.

The desired modular relation is constructed by taking $\mathcal{A} = \mathcal{F}(\mathfrak{q}) $ and $\mathcal{B}= \mathcal{F}(\mathfrak{q})^c $ in \cref{thm:nestomodrel}.
Let us write $\mathcal{K} = \cup_{A \in \mathcal{F}(\mathfrak{q})} (\Orth  A )^c$ and $\mathcal{J} = \cup_{ B \in \mathcal{F}(\mathfrak{q})^c} \Orth  B  $.
We claim that $\mathcal{K} \cup \mathcal{J} = \oPi_n $.

Take, for sake of contradiction, some $\opi \not \in \mathcal{K} \cup \mathcal{J} $.
Note that from $\opi \not\in \mathcal{K}$ we have $\# A_{\opi } = 1 $ for every $A \in \mathcal{F} (\mathfrak{q} )$.
Note as well that from $\opi \not\in \mathcal{J}$ we have that $\#B_{\opi} \neq 1$ for every $B \not\in \mathcal{F} (\mathfrak{q} )$.
Therefore, if $\opi \not \in \mathcal{K} \cup \mathcal{J} $, then $\mathfrak{q} = \mathfrak{p}^{\opi }$, contradicting the assumption that $\mathfrak{q} $ is not a basic hypergraphic polytope.
We obtain that $\mathcal{K} \cup \mathcal{J} = \oPi_n $.
Finally, note that 
$$\mathfrak{q} + \sum\nolimits_{\mathcal{T} \subseteq \mathcal{F}(\mathfrak{q})^c} (-1)^{\# \mathcal{T} }  \left[\mathfrak{q} + \mathcal{F}^{-1} (\mathcal{T} ) \right] \, ,$$
is a modular relation that respects the order $\tilde{\geq}$, showing that the hypotheses of \cref{lm:kernandimdesc} are satisfied.
\end{proof}

For the commutative case we use \cref{lm:kernofcomp}.
Note that we already have a generator set of $\ker \boldsymbol{\Upsilon }_{\mathbf{HGP} }$, so similarly to the proof of \cref{thm:graphcomukernel}, we just need to establish some linear independence.

Recall that two hypergraphic polytopes $\mathfrak{q}_1$ and $\mathfrak{q}_2$ are isomorphic if there is a permutation matrix $P$ such that $x \in \mathfrak{q}_2 \Leftrightarrow x P \in  \mathfrak{q}_1$.
If $\opi_1$ and $\opi_2$ share the same composition type, then $\mathfrak{p}^{\opi_1} $ and $ \mathfrak{p}^{\opi_2} $ are isomorphic, and so we have $\Psi_{\mathbf{HGP}} ( \mathfrak{p}^{\opi_1} ) = \Psi_{\mathbf{HGP} } ( \mathfrak{p}^{\opi_2} )$.
Set $R_{\alpha ( \opi ) } : = \Psi_{\mathbf{HGP} } ( \mathfrak{p}^{\opi} )$ without ambiguity.

\begin{proof}[Proof of \cref{thm:nestocomukernel}]
We use \cref{lm:kernofcomp} with the map $\Psi_{\mathbf{HGP}} = \comu \circ \boldsymbol{\Upsilon }_{\mathbf{HGP}}$.

From the proof of \cref{thm:nestokernel}, to apply \cref{lm:kernofcomp} it is enough to establish that the family $\{R_{\alpha  }  \}_{\alpha \in \mathcal{C}_n }$ is linearly independent.
It would follow that $\ker \Psi_{\mathbf{HGP}} $ is generated by the modular relations, the simple relations and the isomorphism relations, and that $\{ R_{\alpha  } \}_{\alpha \in \mathcal{C}_n }$  is a basis of $\im \Psi_{\mathbf{G}}$, concluding the proof.

To show that $\{R_{\alpha  } \}_{\alpha \in \mathcal{C}_n }$ is linear independent, write $R_{\alpha  } $ on the monomial basis of $QSym$, and use the order $\leq' $ mentioned in \cref{lm:ordercomps}.

As a consequence of \eqref{eq:fundnestocase}, if we write $A_{\opi, \beta} = \#\{\otau \in \oPi_n | \opi \preceq \otau,  \,  \alpha(\otau ) = \beta \} $, from \cref{lm:ordercomps} we have:
\begin{equation}
\begin{split}
R_{\alpha ( \opi ) } = \Psi_{\mathbf{HGP}}(\mathfrak{p}^{\opi  } ) &= \sum_{\opi \preceq \otau} M_{\alpha ( \otau ) }
								 = A_{\opi, \alpha(\opi) } M_{\alpha(\opi)} + \sum_{\alpha(\opi )<' \beta } A_{\opi, \beta } M_{\beta }\, .
\end{split}
\end{equation}
It is clear that $A_{\opi, \alpha(\opi) } > 0 $, so independence follows, which completes the proof.
\end{proof}

\begin{rem}
We have obtained in the proof of \cref{thm:nestocomukernel} that $\{ R_{\alpha}\}_{\alpha \in \mathcal{C}_n}$ is a basis for $Qym_n$.
The proof gives us a recursive way to compute the coefficients $\zeta_{\alpha} $ on the expression $\Psi_{\mathbf{HGP}} ( \mathfrak{q} ) = \sum_{\alpha \in \mathcal{C}_n} \zeta_{\alpha} R_{\alpha }$.
It is then natural to ask if combinatorial properties can be obtained for these coefficients, which are isomorphism invariants.

Similarly, in the non-commutative case, we can write the chromatic quasisymmetric function of a hypergraphic polytope as 
$$\uhsm_{\mathbf{HGP}} ( \mathfrak{q} ) = \sum_{[\opi] \in \oPi_n / \sim} \zeta_{[\opi]}(\mathfrak{q}) \uhsm_{\mathbf{HGP}} (\mathfrak{p}^{[\opi ] })\, , $$ and ask for the combinatorial meaning of the coefficients $( \zeta_{[\opi ]} )_{[\opi ] \in \oPi_n /\sim }$.
These questions are not answered in this paper.
\end{rem}

\subsection{The dimension of \textbf{SC} space}

Let $sc_n := \dim \mathbf{SC}_n$.
Recall that, from \cref{def:SCspace}, the elements of the Hopf algebra $\mathbf{SC}$ are of the form $ \sum\nolimits_{\opi \models [n]  }  \mathbf{M}_{\opi} a_{\opi }$, 
where $a_{\opi_1} = a_{\opi_2 }$ whenever $\opi_1 \sim \opi_2$ in the SC equivalence relation.
Hence, $sc_n$ counts the equivalence classes of $\sim $.

The goal of this section is to compute the asymptotics of $sc_n$, by using the combinatorial description in \cref{prop:combinterpretation}.

\begin{prop}\label{prop:powserim}
Let $F(x) = \sum_{n \geq 0} sc_n \frac{x^n}{n!} $ be the exponential power series enumerating the dimensions of $\mathbf{SC}_n$.
Then 
$$F(x) = \frac{e^x}{ 1 + (1 + x) e^x - e^{2x } }\, . $$
\end{prop}

\begin{prop}\label{prop:growthdimim}
The dimension of $ \mathbf{SC}$ has an asymptotic growth of 
$$  sc_n = n! \gamma^{-n}\left( \tau + o ( \delta^{-n} )  \right)  \, , $$
where $\delta < 1$ is some real number, $\gamma \approxeq 0.814097 \approxeq 1.1745 \log (2) $ is the unique positive root of the equation

$$ e^{2x} = 1 + ( 1+x) e^x \, ,$$
and $\tau = \Res_{\gamma } (F) \approxeq 0.588175 $ is the residue of the function $F$ at $\gamma $.

\end{prop}

In particular, $\dim \mathbf{SC}_n $ is exponentially smaller than $\dim \mathbf{WQSym}_n = \# \oPi_n $, which is asymptotically
$$ n! \log (2)^{-n} \left( \frac{1}{2\log( 2)} + o(1) \right)  \, , $$
according to \cite{barthelemy80}.
Before we prove \cref{prop:powserim,prop:growthdimim}, we introduce a useful combinatorial family.

A barred set composition of $[n]$ is a set composition of $[n]$ where some of the blocks may be barred.
For instance, $13|\overline{45}| 2 $ and $ 12 | 4 | \overline{35} $ are bared set compositions of $\{ 1, 2, 3, 4, 5\}$.
A barred set composition is \textit{integral} if
\begin{itemize}
    \item No two barred blocks occur consecutively, and;
    
    \item Every block of size one is barred;
\end{itemize}

An integral barred set composition is also called an IBSC for short.
In \cref{tbl:smallibsp} we have all the integral barred set compositions of small size:

\begin{table}[ht]
\centering
\begin{tabular}{ l | l | r }
  n & IBSC & Equivalent classes under $\sim $\\ \hline
  0 & $\emptyset$ & $\{ \vec{\emptyset } \} $ \\ \hline
  1 & $\overline{1}$ & $\{ 1 \}$ \\ \hline
  2 & $\overline{12}, 12$ & $ \{1|2, 2|1 \}, \{12 \}$ \\ \hline
  3 & $\overline{123}, \overline{1}|23, \overline{2}|13, \overline{3}|12, $ & $[1| 2| 3)],   \{1| 23 \}, \{2| 13 \}, \{3| 12 \}, $ \\
  &  $ 12|\overline{3}, 13|\overline{2}, 23|\overline{1}, 123$ & $\{12| 3 \}, \{13| 2 \}, \{23| 1 \}, \{123 \} $
\end{tabular}
\caption{Small IBSCs and equivalence classes of $\sim$\label{tbl:smallibsp} }
\end{table}
\vspace{0.2cm}

According to \cref{prop:combinterpretation}, we can construct a map from equivalence classes of $\sim $ and integral barred set compositions: from a set composition, we squeeze all consecutive singletons into one bared block.
This map is a bijection, as is inverted by splitting all bared blocks into singletons, and the equivalence classed obtained is independent on the order that this splitting is done.
So, for instance, $\overline{13}|24 \leftrightarrow \{1|3|24 , 3|1|24 \}$ and $13|24 \leftrightarrow \{13|24 \} $.
See \cref{tbl:smallibsp} for more examples.

\begin{proof}[Proof of \cref{prop:powserim}]
We use the framework developed in \cite{flajolet09} of labeled combinatorial classes.
In the following, a calligraphic style letter denotes a combinatorial class, and the corresponding upper case letter denotes its exponential generating function.
Let $\mathcal{B} $ and $\mathcal{U}$ be the collections $\{\overline{1}, \overline{12}, \cdots  \}$ and $ \{12, 123, 1234, \cdots \} $, respectively, with exponential generating functions $B(x)= e^x - 1 $ and $ U(x) = e^x - 1 - x$.
Additionally, let $\mathcal{O} = \{\emptyset \}$ with $O(x) = 1$.

Let $\mathcal{F}$ be the class of IBSCs, and we denote by $\mathcal{F}^o$ the class of IBSCs that start with an unbarred set.
Denote by $\overline{\mathcal{F}}$ the class of IBSCs that start with a barred set.

Our goal is to show that $F(x) = \frac{e^x}{ 1 + (1 + x) e^x - e^{2x } }$.
By definition we have that $\mathcal{F} =\overline{\mathcal{F}} \sqcup \mathcal{F}^o \sqcup \mathcal{O}$.
Further, we can recursively describe $\overline{\mathcal{F}} $ and $\mathcal{F}^o $ as $\overline{\mathcal{F}} = \mathcal{B} \times (\mathcal{F}^o \sqcup \mathcal{O} ) $ and $\mathcal{F}^o = \mathcal{U} \times \mathcal{F}$.

According to the dictionary rules in \cite{flajolet09}, this implies that $F = \overline{F} + F^o + O$ and that

$$\overline{F} (x) = (e^x - 1 )(F^o(x) +1 ) \, , $$
$$F^o (x) = (e^x - 1  - x)(\overline{F}(x) + F^o(x) +1 ) \, , $$

The unique solution of the system has $F^o(x) = \frac{e^{2x} - (1 + x) e^x}{1 -1 e^{2x} + (x + 1)e^x } $ so it follows
$$F(x) = \frac{ e^x }{1 + (1 + x) e^x - e^{2x } }\, $$ 
as desired.
\end{proof}

With this we can easily compute the dimension of $\mathbf{SC}_n$ for small $n$, and compare it with $\dim \bm{WQSym}_n$, as done in \cref{tbl:scseq}.

\begin{table}
\centering
\begin{tabular}{ l | c | c | c | c | c | c | c | c | c | c  }
 n   & 0 & 1 & 2 & 3 & 4 & 5 & 6 & 7 & 8 & 9\\ \hline
 $sc_n$ & 1 & 1 & 2 & 8 & 40 & \small{242} & \small{1784} & \small{15374} & \small{151008} & \small{1669010} \\ \hline
 $\bm{\pi}_n$ & 1 & 1 & 3 & 13 & 75 & \small{541} & \small{4683} & \small{47293} & \small{545835} & \small{7087261}
\end{tabular}
\caption{First elements of the sequences $sc_n$ and $\boldsymbol{\pi }_n = \dim \bm{WQSym}_n$.\label{tbl:scseq}}
\end{table}

\begin{proof}[Proof of \cref{prop:growthdimim}]
Let $l(x) := e^{2x} - (1 + x) e^x  - 1 $, then $F(x) = -\frac{e^x}{l(x)}$ is the quotient of two entire functions with non-vanishing numerator, so the poles are the zeros of $l(x)$.
Note that $F(x)$ is a counting exponential power series around zero, so it has positive coefficients.
By Pringsheim's Theorem as in \cite{flajolet09}, one of the dominant singularities of $F(x)$ is a positive real number, call it $\gamma $.

We show now that any other singularity $z \neq \gamma$ of $F$, that is, a zero of $l$, has to satisfy $|z|  > |\gamma | $.
Thus, showing that $\gamma $ is the unique dominant singularity and allowing us to compute a simple asymptotic formula.
Suppose, that $z$ is a singularity of $F$ distinct from $\gamma$,  such that $|z| = |\gamma |$.
So, we have that $l(z) = 0 $ and that $z \not\in \mathbb{R}^+$.
The equation $l(z) = 0 $ can easily be rewritten as
$$1 = l(z) + 1 = \sum_{n \geq 1 } z^n \frac{2^n - 1  - n}{n!}  \, . $$
Note that $2^n \geq n+1 $ for $n\geq 1$.
Now we apply the strict triangular inequality on the right hand side to obtain
$$1 < \sum_{n \geq 1 } |z|^n \frac{2^n - n - 1}{n!}  =  \sum_{n \geq 1 } \gamma^n \frac{2^n - n - 1}{n!} = l(\gamma ) + 1 = 1\, , $$
where we note that the inequality is strict for $z\not\in \mathbb{R}^+$ because some of the terms $|z|^n$ do not lie in the same ray through the origin.
This is a contradiction with the assumption that there exists such a pole, as desired.

We additionally prove that $\gamma $ is the unique positive real root, so we can easily approximate it by some numerical method, for instance the bisection method.
The function $l$ in the positive real line satisfies $\lim_{x\to +\infty} l(x) = +\infty $ and $l(0) = -1$, so it has at least one zero.
Note that such zero is unique, as $l'(x)>0 $ for $x$ positive.
Also, since $l'(\gamma ) >0$, the zero $\gamma $ is simple.

Since the function $F(x)$ is meromorphic in $\mathbb{C}$, and $\gamma $ is the dominant singularity, we conclude that 
$$\frac{sc_n}{n!} = \gamma^{-n}\left( \Res_{\gamma }(F) + o ( \delta^{-n} )  \right) \,  ,$$
for any $\delta $ such that $1 > \delta > |\gamma /\gamma_2 | $, where $\gamma_2$ is a second smallest singularity of $F$, if it exists, and arbitrarily large otherwise.

We can approximate $\gamma \approxeq 0.814097 $, and also estimate the residue of $F(x)$ at $\gamma$ as $\tau = \Res_{\gamma } (F) = \frac{e^{\gamma}}{l'(\gamma ) } \approxeq 0.588175 $.
This proves the desired asymptotic formula.
\end{proof}

\subsection{Faces of generalized permutahedra and the singleton commuting equivalence relation}
\label{sec:facesgp}

The main result of this section is the following:

\begin{thm}[Image of $\boldsymbol{ \Upsilon}_{\mathbf{GP}}$]\label{thm:imGP}
The image of $\boldsymbol{ \Upsilon}_{\mathbf{GP}}$ is precisely $\mathbf{SC}$.
\end{thm}

Recall that $\mathbf{SC}$ is the Hopf algebra spanned by $\bigcup_{n\geq 0 } \{\mathbf{N}_{[\opi]}\}_{\opi \in \mathbf{C}_n/_{\sim }}$, a basis indexed by equivalence classes of set compositions on the \textit{singleton commuting equivalence relation}.
This Hopf algebra is precisely the image of $\boldsymbol{ \Upsilon}_{\mathbf{HGP}}$.
Because we have the following inclusion of combinatorial Hopf algebras, $\mathbf{HGP} \subseteq \mathbf{GP}$, it follows that $\mathbf{SC} = \im \boldsymbol{ \Upsilon}_{\mathbf{HGP}}\subseteq \im \boldsymbol{ \Upsilon}_{\mathbf{HGP}}$.

In the remaining of this section we present the proof of the other inclusion.
We remark that an immediate consequence of this result is \cref{cor:PosEmbedding}, where we compare the image of the chromatic map on generalized permutahedra with the image of the chromatic map on posets.
%
%

We first establish a lemma about generalized permutahedra and some other relevant propositions:

\begin{lm}\label{lm:invinGP}
Let $\{a_J\}_{\substack{J\subseteq [n]\\ J \neq \emptyset}}$ be a family of real numbers such that $\mathfrak{q} = \sum_{\substack{J\subseteq [n]\\ J \neq \emptyset}}a_J \mathfrak{s}_J $ is a well defined generalized permutahedron that is a point.
Then, for any set $J$ such that $|J| \geq 2$, we have that $ a_J = 0$.
\end{lm}

We remark that this lemma is trivial for hypergraphic polytopes (it is a simple application of \cref{lm:minkfaces}), and it follows that $\im \boldsymbol{ \Upsilon}_{\mathbf{HGP}} = \mathbf{SC}$.
Before we prove this lemma let us establish first some general claims regarding the coefficients $\{a_J\}_{\substack{J\subseteq [n]\\ J \neq \emptyset}}$.

\begin{prop}\label{prop:moebiusingper}
Consider $n\geq 0 $, and fix a set of real numbers $\{a_J\}_{\substack{J\subseteq [n]\\ J \neq \emptyset}}$.
Define for each non-empty $J \subseteq [n]$ the following
$$ \mathcal U_J = \sum_{\substack{K \cap J \neq \emptyset \\ K \subseteq [n]}} a_K \, , \,  \,  \,  \,  \,  \,  \,  \,  \,  \,  \,  \,  \mathcal W_J = \sum_{\substack{ K \supseteq J \\ K \subseteq [n]}} a_K \, .$$

Then, we have the following relation between $\{\mathcal U_J\}_{\substack{J\subseteq [n]\\ J \neq \emptyset}}$ and $\{\mathcal W_J\}_{\substack{J\subseteq [n]\\ J \neq \emptyset}}$:

$$\mathcal U_J = \sum_{\substack{K \subseteq J\\K \neq \emptyset}} (-1)^{|K| + 1} \mathcal W_K \, .$$ 

Furthermore, for $J$ a singleton, we have that
\begin{equation}\label{eq:singletonuwequation}
\mathcal U_J = \mathcal W_J\, .
\end{equation}
\end{prop}

\begin{proof}
First, \eqref{eq:singletonuwequation} is immediate because we observe that the formulas for $\mathcal U_J$ and $\mathcal W_J$ are the same.

Then, observe that for any finite set $J$, we have that
$$\sum_{K\subseteq X} (-1)^{| K |} = \mathbb{1}[X = \emptyset ] \, . $$

Thus, we have that 
\begin{align*}
\sum_{\substack{K \subseteq J\\K \neq \emptyset}} (-1)^{|K| + 1} \mathcal W_K =& \sum_{\substack{K \subseteq J\\K \neq \emptyset}} \sum_{S \supseteq K } (-1)^{|K| + 1} a_S = \sum_{S \subseteq [n] }\sum_{\substack{K \subseteq J\\ K \subseteq S\\K \neq \emptyset}}(-1)^{|K| + 1} a_S \\
=& \sum_{S \subseteq [n]} a_S \left((-1)^{|\emptyset |}  - \sum_{K \subseteq J\cap S}(-1)^{|K|}  \right)\\
=& \sum_{\substack{S \subseteq [n]\\ S \cap J \neq \emptyset}} a_J = \mathcal U_J\, ,
\end{align*}
as desired.
\end{proof}

\begin{prop}\label{prop:maximalityequations}
Consider $n\geq 0 $, and fix a set of real numbers $\{a_J\}_{\substack{J\subseteq [n]\\ J \neq \emptyset}}$ such that the generalized permutahedron $\mathfrak{q} = \sum_{\substack{J\subseteq [n]\\ J \neq \emptyset}}a_J \mathfrak{s}_J $ is well defined.
Consider $ \mathcal U_J $ as in \cref{prop:moebiusingper}, and for a set $K\subseteq [n]$, let $\vec{e}_K$ be the characteristic vector of $K$, that is if $i\in K$, then $(\vec{e}_K)_i=1$, and $(\vec{e}_K)_i=0$ otherwise.

Then $\mathcal U_K = \max_{x\in \mathfrak{q} } \{ \vec{e}_K^T \cdot x  \} $.
\end{prop}

\begin{proof}

If $\mathfrak{s}_J$ is a simplex, then $ \max_{x\in \mathfrak{s}_J } \{ \vec{e}_K^T \cdot x  \} = \mathbb{1}[J\cap K \neq \emptyset ]$.
Thus, as we have seen in \cref{lm:minkfaces}, optimization problems commute with the Minkowski operations, so we have the following:
\begin{equation}
\begin{split}
\max_{x\in \mathfrak{q} } \{ \vec{e}_K^T \cdot x  \} &= \sum_{\substack{J\subseteq [n]\\ J \neq \emptyset}} a_J \max_{x\in \mathfrak{s}_J } \{ \vec{e}_K^T \cdot x  \} \\
 &= \sum_{\substack{J\subseteq [n]\\ K \cap J \neq \emptyset}} a_J = \mathcal U_K \, ,
\end{split}
\end{equation}
as desired.
\end{proof}

We are now ready to present the proof of \cref{lm:invinGP}.

\begin{proof}[Proof of \cref{lm:invinGP}]

Assume that $ \sum_{\substack{J\subseteq [n]\\ J \neq \emptyset}} a_J \mathfrak{s}_J $ is a well defined generalized permutahedron that is a point, say $\mathfrak{q} = \{\vec{x}\}$.
Define $\mathcal U_J, \mathcal W_J$ as in \cref{prop:moebiusingper}.
We will show that there is no set $J$ such that $\mathcal W_J \neq 0$ and $|J|\geq 2$.
This readily implies that there is no set $J$ such that $a_J \neq 0$ and $|J|\geq 2$, concluding the lemma.

Assume otherwise, by contradiction, that there is some set $J$ such that $\mathcal W_J \neq 0$ and $|J|\geq 2$.
Let $\mathtt{J}_0$ be the smallest such set.
In particular, observe that $\mathcal W_{\mathtt{J}_0} \neq 0$ but $\mathcal W_J= 0 $ for any $J\subsetneq \mathtt{J}_0$ such that $|J| \geq 2$.

Then, from \cref{prop:maximalityequations}, for any set $J$
we have that 
\begin{equation}
\begin{split}
\mathcal U_J &= \max_{x\in \mathfrak{q} }   \{ \vec{e}_J^T \cdot x  \} = \left(\sum_{j\in J} \vec{e}_{\{j\} }^T \right)  \cdot  \vec{x}\\
			 &= \sum_{j\in J} \vec{e}_{\{j\} }^T \cdot  \vec{x} = \sum_{j \in J} \mathcal U_{\{j \}} =   \sum_{j\in J} \mathcal W_{\{j \}} \, .
\end{split}
\end{equation}


Comparing with \cref{prop:moebiusingper}, we have that 
$$ \sum_{\substack{K \subseteq J\\|K| \geq 2}} (-1)^{1+|K|} \mathcal W_K = \mathcal U_J - \sum_{j \in J } \mathcal W_{\{j\}} =  0 \, ,$$
however, if we let $J= \mathtt{J}_0$, we get
$$ \sum_{\substack{K \subseteq \mathtt{J}_0  \\|K| \geq 2}} (-1)^{1+|K|} \mathcal W_K =  (-1)^{1+|\mathtt{J}_0  |} \mathcal W_{\mathtt{J}_0} \neq 0\, . $$
This is a contradiction with the fact that such set $\mathtt{J}_0$ exists, as desired.
\end{proof}

\begin{proof}[Proof of \cref{thm:imGP}]
We know that
$$\boldsymbol{ \Upsilon}_{\mathbf{GP}}(\mathfrak{q} ) = \sum_{\opi \text{ is } \mathfrak{q}-\text{generic}} \mathbf{M}_{\opi} \, . $$

Suppose that $\opi_1 = A_1| \dots |A_k$, and $\opi_2 = B_1| \dots |B_k$ are set compositions such that $\mathfrak{q}_{\opi_1}$ is a point and $\opi_1 \sim \opi_2$.
Define, for $i=1, \dots , k$, the set $F_i = \bigcup_{j=i+1}^k A_j$, and $G_i = \bigcup_{j=i+1}^k B_j$.
Observe that $F_k = G_k = \emptyset$.
From the assumption that $\opi_1 \sim \opi_2$ we get that for a given $i=1, \dots , k$ and $ K\subseteq B_i$ with $|K| \geq 2$, we have that $A_i=B_i$ and $F_i = G_i$.

We wish to show that $\mathfrak{q}_{\opi_2}$ is also a point.
This concludes the proof, because in this way we can group the sum above as
$$\boldsymbol{ \Upsilon}_{\mathbf{GP}}(\mathfrak{q} ) = \sum_{\text{ all } \otau \in [ \opi ] \text{ are } \mathfrak{q}-\text{generic}} \mathbf{N}_{[\opi ]}  \, ,$$
where the sum runs over equivalence classes $[\opi ]$, and this is trivially an element of $\mathbf{SC}$.
%
%

We can rearrange the sum obtained in \cref{eq:facepermutahedra} as follows: for each $i=1, \dots k$ and non-empty $K \subseteq A_i$, we group together all the sets $I$ such that $I_{\opi} = K$.
Those are precisely all the sets $I=J \cup K$ for some $J\subseteq F_i$.
Thus, we have that
$$ \mathfrak{q}_{\opi_1} = \sum_{i=1}^k \sum_{\substack{K\subseteq A_i \\ K \neq \emptyset}} \mathfrak{s}_K \left( \sum_{J\subseteq F_i } a_{J\cup K} \right)\, . $$

From \cref{lm:invinGP}, we have that 
\begin{equation}\label{eq:LASTEQUATION}
\sum_{J\subseteq F_i } a_{J\cup K} = 0 ,  \text{ for each $i=1, \dots , k$ and each $K\subseteq A_i$ with $|K| \geq 2$.}
\end{equation}

Similarly, we have that 
\begin{equation}
 \mathfrak{q}_{\opi_2} =\sum_{i=1}^k \sum_{\substack{K\subseteq B_i \\ K \neq \emptyset }} \mathfrak{s}_K \left( \sum_{J\subseteq G_i } a_{J\cup K} \right)\, . 
\end{equation}
The proof is concluded when we establish that $\sum_{J\subseteq G_i } a_{J\cup K} = 0 $ for each $i=1, \dots , k$ and each $K\subseteq B_i$ with $|K| \geq 2$.
This is precisely \eqref{eq:LASTEQUATION} because in this case we have that$A_i=B_i$ and $F_i = G_i$.
Therefore $ \mathfrak{q}_{\opi_2}$ is a point, as desired.
\end{proof}

\section{Hopf species and the non-commutative universal property\label{ch:HM}}

In \cite{aguiar06}, a character in a Hopf algebra is defined as a multiplicative linear map that preserves unit, and a combinatorial Hopf algebra (or CHA, for short) is a Hopf algebra endowed with a character.
For instance, a character $\eta_0$ in $QSym $ is $\eta_0(M_{\alpha }) = \mathbb{1}[l(\alpha ) = 1 ]$.
In fact, the CHA of quasisymmetric functions $(QSym, \eta_0)$ is a terminal object in the category of CHAs, i.e. for each CHA $(h, \eta) $ there is a unique combinatorial Hopf algebra morphism $\Psi_{h}:h \to QSym $.

Our goal here is to draw a parallel for Hopf monoids in vector species.
We see that the Hopf species $\hs{WQSym}$ plays the role of $QSym$.
Specifically, we construct a unique Hopf monoid morphism from any combinatorial Hopf monoid $\overline{h}$ to $\hs{WQSym}$, in line with what was done in \cite{aguiar06} and \cite{white16hopf}.

In the last section we investigate the consequence of this universal property on the Hopf structure of hypergraphic polytopes and posets.
We use \cref{thm:nestokernel} and \cref{prop:growthdimim} to obtain that no combinatorial Hopf monoid morphism from $\hs{HGP}$ to $\hs{Pos}$ exists.

\begin{rem}
The category of combinatorial Hopf monoids was introduced in two distinct ways, by \cite{aguiar17} in vector species, and by \cite{white16hopf} in pointed set species, which we call here a \textit{comonoidal combinatorial Hopf monoid}.
Here we consider the notion of \cite{aguiar17}.

In \cite{white16hopf}, White shows that a comonoidal combinatorial Hopf monoid in coloring problems is a terminal object on the category of CCHM.
Nevertheless, it is already advanced there that, if we consider a weaker notion of combinatorial Hopf monoid, the terminal object in such category is indexed by set compositions.
No counterpart of $\overline{\mathbf{WQSym}}$ in pointed set species is discussed here.
\end{rem}


\subsection{Hopf monoids in vector species}

In this section, we recall the basic notions on Hopf monoids in vector species introduced in \cite[Chapter 8]{aguiar10}.
We write $Set^{\times }$ for the category of finite sets with bijections as only morphisms, and write $Vec_{\mathbb{K}}$ for the category of vector spaces over $\mathbb{K}$ with linear maps as morphisms.
A \textit{vector species}, or simply a species, is a functor $\overline{a}: Set^{\times} \to Vec_{\mathbb{K}}$.
Species forms a category $Sp_{\mathbb{K}}$, where functors are natural transformations between species.
For a species $\overline{a}$, we denote by $\overline{a}[I]$ the vector space mapped from $I$ through $\overline{a}$.
For a natural transformation $\eta: \overline{a} \Rightarrow \overline{b}$, we may write either $\eta[I]$ or $\eta_I $ to the corresponding map $\overline{a}[I] \to \overline{b}[I]$.

The Cauchy product is defined on species $\overline{a}, \overline{b}$ as follows:
$$(\overline{a} \cdot \overline{b})[I] = \bigoplus_{I = S \sqcup T} \overline{a}[S]\otimes \overline{b}[T] \, .$$

Two fundamental species are of interest.
The first one $\overline{I} $ acts as the identity for the Cauchy product, and is defined as $\overline{I}[\emptyset ] = \mathbb{K}$ for the empty set, and $ \overline{I}[A ] = 0$ otherwise. The functor $\overline{I}$ maps morphisms to the identity.
The exponential species $\overline{E}$ is defined as $\overline{E}[A] = \mathbb{K}$ for any set $A$, and maps morphisms to the identity as well.

A vector species $\overline{a}$ is called a bimonoid if there are natural transformations $\mu:\overline{a}\cdot \overline{a} \Rightarrow \overline{a}$, $\iota: \overline{I} \Rightarrow \overline{a}$, $\Delta : \overline{a} \Rightarrow \overline{a} \cdot \overline{a}$ and $\epsilon: \overline{a} \Rightarrow \overline{I} $ that satisfy some properties which we recover here only informally.
We address the reader to \cite[Section 8.2 - 8.3]{aguiar10} for a detailed introduction of bimonoids in species.

\begin{itemize}

\item The natural transformation $\mu $ is associative.

\item The natural transformation $\iota$ acts as unit on both sides.

\item The natural transformation $\Delta $ is coassociative.

\item The natural transformation $\epsilon $ acts as a counit on both sides.

\item Both $\mu, \Delta $ are determined by maps 
$$\mu_{A, B}: \overline{a}[A]\otimes \overline{a}[B] \to \overline{a}[A\sqcup B]\, , $$
$$\Delta_{A, B}: \overline{a}[A\sqcup B] \to \overline{a}[A]\otimes \overline{a}[B] \, . $$

\item The natural transformations satisfy some coherence relations typical for Hopf algebras. In particular it satisfies diagram \ref{cd:bimonoid} below, which enforces that the multiplicative and comultiplicative structure agree.

\begin{equation}\label{cd:bimonoid}
\begin{tikzcd}
\overline{a}[I]\otimes \overline{a}[J]
\arrow[rrr, "{\mu_{I, J}}"] 
\arrow[d, "{\Delta_{R, T} \otimes \Delta_{U, V}}"] 
& & &
\overline{a}[S]
\arrow[d, "{ \Delta_{M, N} }"] \\
\overline{a}[R]\otimes \overline{a}[T] \otimes \overline{a}[U]\otimes \overline{a}[V]
\arrow[rrr, "{ (\mu_{R, U} \otimes \mu_{T, V})\circ \twist }"]
& & &
\overline{a}[M]\otimes \overline{a}[N]
\end{tikzcd}
\end{equation}
\end{itemize}

We consider the canonical isomorphism $\beta : V \otimes W \to W \otimes V $, and also refer to any composition of tensors of identity maps and $\beta $ as a \textit{twist}.
Whenever needed, we consider a suitable twist function without defining it explicitly, by letting the source and the target of the map clarify its precise definition.
For instance, that is done above in Diagram \eqref{cd:bimonoid}.

We use $\mu_{A, B} $ and $\cdot_{A, B}$ interchangeably for the monoidal product.
Namely, $\cdot_{A, B}$ will be employed for in line notation

We say that a bimonoid $\overline{h}$ is \textit{connected} if the dimension of $\overline{h}[\emptyset] $ is one.
A bimonoid is called a \textit{Hopf monoid} if there is a natural transformation $s:\overline{h} \Rightarrow \overline{h}$, called the \textit{antipode}, that satisfies 
$$ \mu \circ(id_{\overline{h}} \cdot s)\circ\Delta = \iota \circ \epsilon = \mu \circ(s  \cdot id_{ \overline{h}})\circ \Delta \, .$$

\begin{prop}[Proposition 8.10 in \cite{aguiar10}]\label{prop:takeuchi}
If $\overline{h}$ is a connected bimonoid, then there is an antipode on $\overline{h}$ that makes it a Hopf monoid.
\end{prop}

\begin{smpl}[Hopf monoids]\leavevmode
\begin{itemize}
\item The exponential vector species can be endowed with a trivial product and coproduct. This is a connected bimonoid, hence it is a Hopf monoid.

\item The vector space $\hs{G}[I]$ (resp. $\hs{Pos}, \hs{GP}$ and $\hs{HGP} $ has a basis given by all graphs on the vertex set $I$ (resp. partial orders on the set $I$, generalized permutahedra in $\mathbb{R}^I$, hypergraphic polytopes in $\mathbb{R}^I$) defines a Hopf monoid with the operations introduced in \cref{ch2}.

\end{itemize}
\end{smpl}

\subsection{Combinatorial Hopf monoids}

The notion of characters in Hopf monoids was already brought to light in \cite{aguiar17}, where it is used to settle, for instance, a conjecture of Humpert and Martin \cite{humpert12} on graphs.

\begin{defin}
Let $\overline{h}$ be a Hopf monoid.
A \textit{Hopf monoid character} $\eta:\overline{h} \Rightarrow \overline{E}$, or simply a \textit{character}, is a monoid morphism such that $\eta_{\emptyset } = \epsilon_{\emptyset }$ and the following diagram commutes:
\begin{equation}\label{cd:chars}
\begin{tikzcd}
\overline{h}[I]\otimes \overline{h}[J]
\arrow[r, "{\mu_{I, J}}"] 
\arrow[d, "{\eta_I \otimes \eta_J}"] 
& 
\overline{h}[A]
\arrow[d, "{ \eta_A }"] \\
\mathbb{K} \otimes \mathbb{K}
\arrow[r, "{ \cong }"]
&
\mathbb{K}
\end{tikzcd}
\end{equation}

A \textit{combinatorial Hopf monoid} is a pair $(\overline{h}, \eta)$ where $\overline{h} $ is a Hopf monoid, and $\eta$ a character of $\overline{h}$.
\end{defin}

The condition that $\eta $ and $\epsilon $ coincide in the $\emptyset $ level is commonly verified in Hopf monoids of combinatorial objects.
In particular, this condition is always verified in connected Hopf monoids.

\begin{smpl}[Combinatorial Hopf monoids]\label{smpl:CHS}
From the examples on Hopf monoids above and the characters defined in \cref{ch2}, we can construct combinatorial Hopf monoids: in $\hs{G}$ with the character $\eta(G)= \mathbb{1}[G \text{ has no edges}]$, in $\hs{Pos}$ with the character $\eta(P) = \mathbb{1}[P \text{ is antichain}]$, and in $\hs{GP}$ with the character given by $\eta(\mathfrak{q}) = \mathbb{1}[\mathfrak{q} \text{ is a point }]$.
\end{smpl}

A combinatorial Hopf monoid morphism $\alpha:(\overline{h}_1, \eta_1) \Rightarrow (\overline{h}_2, \eta_2) $ is a Hopf monoid morphism $\alpha: \overline{h}_1 \Rightarrow \overline{h}_2$ such that the following diagram commutes:

\begin{equation}\label{cd:chamorph}
\begin{tikzcd}
\overline{h}_1
\arrow[rr, bend left, "{\alpha }"]
\arrow[rd, "{\eta_1}"]
&
&
\overline{h}_2
\arrow[ld, "{\eta_2}"]\\
&
\overline{E}
&
\end{tikzcd}
\end{equation}

We introduce the Fock functors, that give us a construction of several graded Hopf algebras from a Hopf monoid and, more generally, construct graded vector spaces from vector species.
The topic is carefully developed in \cite[Section 3.1, Section 15.1]{aguiar10}.

\begin{defin}[Fock functors]
Denote by $gVec_{\mathbb{K}}$ the category of graded vector spaces over $\mathbb{K}$.
We focus on the following Fock functors $\mathcal{K}, \overline{\mathcal{K}}: Sp \to gVec_{\mathbb{K}}$, called full Fock functor and bosonic Fock functor, respectively, defined as:

$$\fff(q) := \bigoplus_{n\geq 0} q [ \{1, \dots , n\} ] \text{  and  } \bff(q) :=  \bigoplus_{n\geq 0} q [\{1, \dots , n\} ]_{S_n}\, , $$
where $V_{S_n}$ stands for the vector space of coinvariants on $V$ over the action of $S_n$, i.e. the quotient of $V$ under all relations of the form $x - \sigma(x)$, for $\sigma \in S_n$.

If $\overline{h}$ is a combinatorial Hopf monoid with structure morphisms $\mu, \iota, \Delta, \epsilon$, then $\fff(\overline{h}) $ and $\bff(\overline{h} )$ are Hopf algebras with related structure maps.
If $\eta $ is a character of $\overline{h}$, then $\fff(\overline{h}) $ and $\bff(\overline{h} )$ also have a character.
\end{defin}

\begin{smpl}[Fock functors of some Hopf monoids]\leavevmode
\begin{itemize}

\item The Hopf algebra $\fff(\overline{I})$ is the linear Hopf algebra $\mathbb{K}$.
The Hopf algebra $\fff(\overline{E})$ is the polynomial Hopf algebra $\mathbb{K}[x]$.

\item The Hopf algebras $\fff(\hs{G} )$, $\fff(\hs{Pos})$ and $\fff(\hs{GP})$ are the Hopf algebras of graphs $\mathbf{G}$, of posets $\mathbf{Pos}$ and of generalized permutahedra $\mathbf{GP}$ introduced above.

\item The Hopf algebra $\fff(\hs{HGP} )$ is the Hopf algebra $\mathbf{HGP}$, the Hopf subalgebra of $\mathbf{GP}$ introduced above.

\end{itemize}
\end{smpl}

\subsection{The word quasisymmetric function combinatorial Hopf monoid}

Recall that a coloring of a set $I$ is a map $f: I  \to \mathbb{N}$, and that $\mathfrak{C}_I$ be the set of colorings of $I$.
Recall as well that a set composition $\opi = S_1| \cdots | S_l$ can be identified with a total preorder $R_{\opi} $, where we say $a\, R_{\opi}\, b $ if $a\in S_i$ and $b\in S_j$ satisfy $i \leq j$.
For a set composition $\opi$ of $A$ and a non-empty subset $I \subseteq A$, we define $\opi|_I$ as the 
set composition of $I$ obtained by restricting the preorder $R_{\opi} $ to $I$.

If $I, J$ are disjoint sets, and $f\in \mathfrak{C}_I$ and $g\in \mathfrak{C}_J$, then we set $f*g \in \mathfrak{C}_{I \sqcup J}$ as the unique coloring in $I\sqcup J$ that satisfies both $f*g|_I = f$ and $f*g|_J = g$.

For a set composition $\opi \in \oPi_I$, let $\mathbb{M}_{\opi} = \sum\nolimits_{\substack{f \in \mathfrak{C}_I \\ \opi(f) = \opi}} [f] $ be a formal sum of colorings, and define $\hs{WQSym }[I]$ as the span of $\{ \mathbb{M}_{\opi } \}_{\opi \in \oPi_I}$.
This gives us a $\mathbb{K}$-linear space with basis enumerated by $\oPi_I$, so that $\hs{WQSym} $ is a species.

Further, define the monoidal product operation with 
\begin{equation}\label{eq:multwqsym}
\mathbb{M}_{\opi} \cdot_{A, B} \mathbb{M}_{\otau} = \sum_{\substack{f \in \mathfrak{C}_A \\ \opi(f) = \opi}} \sum_{\substack{g \in \mathfrak{C}_B \\ \opi(g) = \otau}} [f*g] = \sum_{\substack{\olambda|_A = \opi \\ \olambda|_B = \otau}} \mathbb{M}_{\olambda }\, .
\end{equation}

We write $I \, <_{\opi} \, J$ whenever there is no $i \in I$ and $j \in J$ such that $j \, R_{\opi } \, i$.
The coproduct $\Delta_{I, J} M_{\opi } $ is defined as
$$\mathbb{M}_{\opi|_I} \otimes \mathbb{M}_{\opi|_J}\, ,$$ 
whenever $I \, <_{\opi} \, J$, and is zero otherwise.

If we set the unit as $\iota_{\emptyset}(1) = \mathbb{M}_{\emptyset } $ and the counit acting on the basis as $\epsilon( \mathbb{M}_{\opi  })  = \mathbb{1}[\opi = \emptyset]$ we get a Hopf monoid.
In fact, this is the dual Hopf monoid of faces $\Sigma^* =\Sigma_1^* $ in \cite{aguiar10}.

\begin{prop}[\citeaguHM  ]
With these operations, the vector species $\hs{WQSym}$ becomes a Hopf monoid.
\end{prop}


\begin{prop}[\citeaguWQ]
The Hopf algebra $\mathcal{K}(\hs{WQSym})$ is the Hopf algebra on word quasisymmetric functions $\mathbf{WQSym}$, and $\bff(\hs{WQSym})$ is the Hopf algebra on quasisymmetric functions $QSym$.
\end{prop}

The identification is as follows: $\mathcal{K}(\hs{WQSym})$ and $\mathbf{WQSym}$ by identifying a coloring $f:[n] \to \mathbb{N}$ with the noncommutative monomial $\prod_{i=1}^n a_{f(i)} =: a_f$, and extend this to identify $\mathbb{M}_{\opi } $ with $\mathbf{M}_{\opi}$.

\begin{prop}[Combinatorial Hopf monoid on $\hs{WQSym}$]
Take the character $\eta:\hs{WQSym} \Rightarrow \lhs$ defined in the basis elements as
\begin{equation}\label{eq:charWQdefin}
\eta_0[I](\mathbb{M}_{\opi}) = \mathbb{1}[l(\opi) \leq 1]\, .
\end{equation}

This turns $(\hs{WQSym}, \eta_0)$ into a combinatorial Hopf monoid.
\end{prop}
\begin{proof}
We write $\eta_{0, I} = \eta_0[I]$ for short.
That $\eta_0$ is a natural transformation is trivial, and also $\eta_{0, \emptyset}(\mathbb{M}_{ \emptyset}) = 1$, so it preserves the unit.

To show that $\eta_0$ is multiplicative, we just need to check that the diagram \eqref{cd:chars} commutes for the basis elements, i.e. if $A = I \sqcup J$, then
\begin{equation}\label{eq:proofCHA}
\eta_{0, I}(\mathbb{M}_{\opi }) \eta_{0, J}(\mathbb{M}_{\otau })  = \eta_{0, A}(\mathbb{M}_{\opi }\mathbb{M}_{\otau } ) = \sum_{\substack{\ogamma \in \oPi_A \\ \ogamma|_I = \opi \\  \ogamma|_J = \otau }} \eta_{0, A}(\mathbb{M}_{\ogamma }) \, .
\end{equation}

Note that if $\ogamma $ is a set composition of $A$ such that $\ogamma|_I = \opi $, then trivially we have that $l(\opi ) \leq l(\ogamma )$, so from \eqref{eq:charWQdefin}, $\eta_{0, I}(\mathbb{M}_{\opi }) = 0 \Rightarrow \eta_{0, A}(\mathbb{M}_{\ogamma } ) = 0 $.
Similarly, if $\ogamma|_J = \otau$, we have $\eta_{0, J}(\mathbb{M}_{\otau }) = 0 \Rightarrow \eta_{0, A}(\mathbb{M}_{\ogamma } ) = 0$.

So it is enough to consider the case where $\eta_{0, I}(\mathbb{M}_{\opi })= \eta_{0, J}(\mathbb{M}_{\otau }) = 1 $, i.e. $l(\opi), l(\otau) \leq 1$.
Now, if $\gamma $ has only one part, it does indeed hold that $\gamma|_I = \opi $ and $\gamma|_J = \otau $, so there is a unique $\ogamma $ on the right hand side of \eqref{eq:proofCHA} that satisfies $l(\ogamma ) \leq 1$, and this concludes the proof.
\end{proof}

\subsection{Universality of WQSym}

The following theorem is the main theorem of this section.
For connected Hopf monoids, this is a corollary of \cite[Theorem 11.23]{aguiar10}.

\begin{thm}[Terminal object in combinatorial Hopf monoids]\label{thm:universalhs}
Given a Hopf monoid $\overline{h} $ with a character $\eta: \overline{h} \Rightarrow \lhs$, there is a unique combinatorial Hopf monoid morphism $\uhsm_{\overline{h}}:\overline{h} \Rightarrow \hs{WQSym}$, i.e. a unique Hopf monoid morphism $\uhsm_{\overline{h}}$ such that the following diagram commutes:

\begin{equation}\label{cd:uhsm}
\begin{tikzcd}
\overline{h}
\arrow[rr, "{\boldsymbol{\Upsilon }_{\overline{h}} }"] 
\arrow[rd, "{ \eta }"] 
& 
&
\hs{WQSym}
\arrow[dl, "{ \eta_0 }"] \\
&
\lhs
&
\end{tikzcd}
\end{equation}
\end{thm}

We remark that this is a claim motivated in \cite[Theorem 11.23]{aguiar10}, which applies to any connected Hopf monoid.
There, the notion of \textit{positive monoid} was introduced, a monoid in species such that $h[\emptyset] = 0$ and with no unit axioms.
Any Hopf monoid $h$ can become a positive monoid $h^+$ by setting $h^+[\emptyset] = 0 $ and $h^+[I]=h[I]$ for any non-empty set $I$.
A functor $\mathcal T^{\vee}$, mapping positive monoids to Hopf monoids, was constructed, so that any \textit{positive monoid} $q$, connected Hopf monoid $h$ and monoid morphism $\eta: h^+ \Rightarrow q$, there exists a unique Hopf monoid morphism $\eta_{\vee}:h \Rightarrow \mathcal T^{\vee}(q)$ with the following commuting diagram on positive monoids:
\begin{equation}\label{cd:universalcommutativediagaguiar}
\begin{tikzcd}
h^+
\arrow[rr, "{\eta^+_{\vee} }"] 
\arrow[rd, "{ \eta }"] 
& 
&
\mathcal T^{\vee}(q)^+
\arrow[dl, "{\epsilon(q) }"] \\
&
q
&
\end{tikzcd}
\end{equation}
where $\epsilon(q) :\mathcal T^{\vee}(q)^+ \Rightarrow q $ is a map that comes from the construction of $\mathcal T^{\vee}$.
In the case where $q$ is the positive exponential monoid, the resulting Hopf monoid $\mathcal T^{\vee}(q) $ is precisely $\hs{WQSym}$, thus obtaining \cref{thm:universalhs} for connected Hopf monoids.

In fact, the result presented in \cref{thm:universalhs} is a minor extension of \cite[Theorem 11.23]{aguiar10} to Hopf monoids that are not necessarily connected, but whose character agrees with the counit in $h[\emptyset]$.
First we will present a self contained proof by means of multi-characters.
In \cref{rem:aguiarwasright}, we present a more direct proof, using \cite[Theorem 11.23]{aguiar10} and a suitably constructed connected Hopf monoid.
This proof was kindly pointed out by a reviewer.


\begin{defin}[Multi-character and other notations]\label{defin:multichar}
For a set composition on a non-empty set $I$, say $\opi=S_1 | \cdots | S_l$ with $k\ 0$, denote for short 
$$\overline{h}[\opi] = \bigotimes_{i=1}^l \overline{h}[S_i]\, ,$$
and similarly define for a natural transformation $\zeta:\overline{h} \Rightarrow \overline{b}$ the linear transformation $\zeta[\opi]:\overline{h}[\opi] \to \overline{b}[\opi]$ as $\zeta[\opi] = \bigotimes_{i=1}^l \zeta[S_i] $.
For a character, $\zeta[\opi ]:\overline{h}[\opi ] \to \mathbb{K}^{\otimes l} \cong \mathbb{K}$.

For a set composition $\opi$ on $I$ of length $k$, let us define $\Delta_{\opi} $ as a map 
$$\Delta_{\opi} :\overline{h}[I] \to \overline{h}[\opi]\, ,$$
inductively as follows:
\begin{itemize}

\item If the length of $\opi$ is 1, then $\Delta_{\opi}= id_{\overline{h}[I]}$.

\item If $ \opi = S_1 | \cdots | S_k$ for $k>1$, let $\otau = S_1| \cdots | S_{k-1}$ and define
\begin{equation}\label{eq:deltaopi}
\Delta_{\opi} =  ( \Delta_{\otau } \otimes id_{S_k} ) \circ \Delta_{ I \setminus S_k, S_k} \, .
\end{equation} 
\end{itemize}

Note that, by coassociativity, this definition of $\Delta_{\opi} $ is independent of the chosen order in the inductive definition in \eqref{eq:deltaopi}, i.e. for any set $I=A \sqcup B $ and set composition $\opi \in \oPi_I$ such that $A <_{\opi } B$, we have 
\begin{equation}\label{eq:multicoass}
\Delta_{\opi} = (\Delta_{\opi|_A} \otimes \Delta_{\opi|_B} ) \circ \Delta_{A, B} \, .
\end{equation}
We can define $f_{\opi, \eta }: \overline{h}[I] \to \mathbb{K}$ as

$$\overline{h}[I] \xrightarrow{\Delta_{\opi}} \overline{h}[\opi] \xrightarrow{\eta[\opi] } \mathbb{K}^{\otimes l} \cong \mathbb{K} \, ,$$

Finally, if $A = I \sqcup J $ with $I\neq \emptyset\neq J$ and $\opi \in \oPi_I, \otau\in\oPi_J$, then we write both $\opi| \otau$ and $(\opi, \otau)$ for the unique set composition $\ogamma \in \oPi_A$ such that $\ogamma|_I = \opi$, $\ogamma|_J = \otau$, and $I <_{\ogamma} J$. 
\end{defin}

\begin{smpl}
In the graph combinatorial Hopf monoid $\hs{G}$, take the labeled cycle $C_5$ on $\{ 1, 2, 3, 4, 5\}$ given in \cref{fig:ciclein5}.
Denote by $K_J$ the complete graph on the labels $J$ and by $0_J$ the empty graph on the labels $J$.

\begin{figure}[ht]
    \centering
    \includegraphics[scale=0.5]{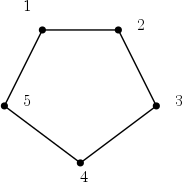}
    \caption{Cycle on the set $[5]$ \label{fig:ciclein5}}
\end{figure}

Consider the set compositions $\opi_1 = 13 | 2| 45$ and $\opi_2 = 24|13|5 $.
Then
$$\Delta_{\opi_1} (C_5 )  = 0_{\{ 1, 3 \} } \otimes 0_{\{ 2\} } \otimes K_{\{ 4, 5\} } \, , $$
$$\Delta_{\opi_2} (C_5 )  = 0_{\{ 2, 4 \} } \otimes 0_{\{ 1, 3\} } \otimes 0_{\{ 5 \} } \, , $$
in particular, $f_{\opi_1, \eta} ( C_5) = 0 $ and $f_{\opi_2, \eta} ( C_5) = 1 $.
Generally, 
\begin{equation}\label{eq:staplgraph}
f_{\opi, \eta} ( G) = \mathbb{1}[\makepar(\opi ) \text{ is a stable set partition on } G]\, .
\end{equation}

From \eqref{eq:staplgraph} and from \cref{lm:monbasis} we have that $\fff ( \uhsm_{\hs{\gHa}} ) = \uhsm_{\gHa} $ is the chromatic symmetric function in non-commutative variables, and that $\bff (\uhsm_{\hs{\gHa}} ) = \Psi_{\gHa}$ is the chromatic symmetric function.

In a similar way we can establish that $\fff ( \uhsm_{\hs{Pos}} ) = \uhsm_{\mathbf{Pos}} $, that $\bff (\uhsm_{\hs{Pos}} ) = \Psi_{\mathbf{Pos}}$, that $\fff ( \uhsm_{\hs{\gpHa}} ) = \uhsm_{\gpHa} $ and that $\bff (\uhsm_{\hs{\gpHa}} ) = \Psi_{\gpHa}$.

\end{smpl}

With this notation, we can rephrase diagram \eqref{cd:bimonoid} in a different way:

\begin{prop}\label{prop:4diaggen}
Consider a Hopf monoid $(\overline{h}, \mu, \iota, \Delta, \epsilon)$.
Let $\ogamma = C_1 | \dots | C_k $ be a set composition on $S = I\sqcup J $, where $I, J$ are non-empty sets.
Write $A_i := C_i \cap I $ and $B_i := C_i \cap J$, and let $\opi := ( \ogamma|_I , \ogamma|_J ) = A_1| \dots |A_k|B_1 | \dots | B_k $, erasing the empty blocks.

Define $\mu_{(\ogamma, I, J)} : \overline{h}[\opi] \to \overline{h}[\ogamma]$ as the tensor product of the maps 
$$ \overline{h}[A_i] \otimes \overline{h}[B_i] \xrightarrow{\mu_{A_i, B_i}} \overline{h}[C_i] \, ,$$
composed with the necessary twist so that it maps $\overline{h}[\opi ] \to \overline{h}[\ogamma ]$.
Then the following diagram commutes:

\begin{equation}\label{cd:genbim}
\begin{tikzcd}
\overline{h}[I] \otimes \overline{h}[J]
\arrow[rr, "{\mu_{I, J} }"] 
\arrow[d, "{ (\Delta_{\ogamma|_I}\otimes \Delta_{\ogamma|_J} )}"] 
&
&
\overline{h}[S]
\arrow[d, "{ \Delta_{\ogamma} }"] \\
\overline{h}[\opi]
\arrow[rr, "{\mu_{(\ogamma, I, J)}}"]
&
&
\overline{h}[\ogamma]
\end{tikzcd}
\end{equation}
\end{prop}

Note that diagram \eqref{cd:bimonoid} corresponds to diagram \eqref{cd:genbim} when $k=2$.
We prove now that Diagram \eqref{cd:genbim} is obtained by gluing diagrams of the form of Diagram \eqref{cd:bimonoid}:

\begin{proof}
We act by induction on the length of $\ogamma$, $k := l(\ogamma )$.
The base case is for $k=2 $, where we recover diagram \eqref{cd:bimonoid}.

Suppose now that $k\geq 3$.
Applying $I = A_1 \sqcup A_2, J= B_1 \sqcup B_2$ to \eqref{cd:bimonoid} we have the following commuting diagram:

\begin{equation}\label{cd:leftmostdiag4diaggen}
\begin{tikzcd}
\overline{h}[A_1\sqcup A_2 | B_1 \sqcup B_2 ]
\arrow[rrr, "{\mu_{A_1\sqcup A_2, B_1\sqcup B_2} }"] 
\arrow[d, "{ \Delta_{A_1, A_2} \otimes \Delta_{B_1, B_2} }"] 
&
&
&
\overline{h}[C_1 \sqcup C_2]
\arrow[d, "{ \Delta_{C_1, C_2} }"] \\
\overline{h}[A_1|A_2|B_1|B_2]
\arrow[rrr, "{(\mu_{A_1, B_1}\otimes \mu_{A_2, B_2} )\circ \twist }"]
&
&
&
\overline{h}[C_1|C_2]
\end{tikzcd}
\end{equation}

Write $I'=I\setminus (C_1 \sqcup C_2) $ and $J' = J \setminus (C_1 \sqcup C_2) $, let $\ogamma' = C_3 | \dots | C_l $ and take $\ogamma^o = C_1 \sqcup C_2 | C_3 | \dots | C_l = \{ C_1 \sqcup C_2 \} | \gamma' $.
Observe that $\gamma ' $ is a partition of a non-empty set.
By tensoring diagram \eqref{cd:leftmostdiag4diaggen} with

\begin{equation}\label{cd:idmap}
\begin{tikzcd}
\overline{h}[A_3|B_3|A_4| \dots |B_l]
\arrow[rrr, "{ \mu_{(\ogamma', I', J')} }"] 
\arrow[d, "{ \id }"] 
&
&
&
\overline{h}[C_3| \dots | C_l]
\arrow[d, "{ \id }"] \\
\overline{h}[A_3|B_3|A_4| \dots |B_l]
\arrow[rrr, "{ \mu_{(\ogamma', I', J')} }"]
&
&
&
\overline{h}[C_3| \dots | C_l]
\end{tikzcd}
\end{equation}
we have:

\begin{equation}\label{cd:lowerdiag4diaggen}
\begin{tikzcd}
\overline{h}[A_1\sqcup A_2 | B_1 \sqcup B_2  | A_3 | B_3 | \dots ]
\arrow[rrrrr, "{\mu_{A_1\sqcup A_2, B_1\sqcup B_2} \otimes \mu_{(\ogamma', I', J')} }"] 
\arrow[d, "{ \Delta_{A_1, A_2} \otimes \Delta_{B_1, B_2} \otimes id }"] 
&
&
&
&
&
\overline{h}[\ogamma^o ]
\arrow[d, "{ \Delta_{C_1, C_2} \otimes id }"] \\
\overline{h}[A_1|A_2|B_1|B_2 | A_3 | B_3 | A_4 | \dots ]
\arrow[rrrrr, "{(\mu_{A_1, B_1}\otimes \mu_{A_2, B_2} \otimes \mu_{(\ogamma', I', J')} )\circ \twist }"]
&
&
&
&
&
\overline{h}[\ogamma]
\end{tikzcd}
\end{equation}

Note that 
$$\mu_{A_1, B_1} \otimes \mu_{A_2, B_2} \otimes \mu_{(\ogamma', I', J' )} = \mu_{(\ogamma, I, J)}\, ,$$
$$\mu_{A_1\sqcup A_2, B_1\sqcup B_2} \otimes \mu_{(\ogamma', I', J' )} = \mu_{(\ogamma^o, I, J)}\, .$$
So, by induction hypothesis, \eqref{cd:genbim} commutes for the set composition $\ogamma^o = C_1 \sqcup C_2 | C_3 | \dots | C_l $.
Apply the necessary twists so as to glue with with diagram \eqref{cd:lowerdiag4diaggen} as follows:

\begin{equation}\label{cd:concludeproof4diaggen}
\begin{tikzcd}
\overline{h}[I] \otimes \overline{h}[J]
\arrow[rrr, "{\mu_{I, J} }"] 
\arrow[d, "{ \twist \circ (\Delta_{\ogamma^o|_I}\otimes \Delta_{\ogamma^o|_J} ) }"] 
&
&
&
\overline{h}[S]
\arrow[d, "{ \Delta_{\ogamma^o} }"] \\
\overline{h}[A_1\sqcup A_2|B_1\sqcup B_2| A_3| B_3| \dots  | B_l]
\arrow[rrr, "{\mu_{(\ogamma^o, I, J)}}"]
\arrow[d, "{\twist \circ (\Delta_{A_1, A_2}\otimes \Delta_{B_1, B_2} \otimes id)}"]
&
&
&
\overline{h}[\ogamma']
\arrow[d, "{\Delta_{C_1, C_2} \otimes id}"] \\
\overline{h}[A_1|B_1|A_2| \dots  | B_l]
\arrow[rrr, "{\mu_{(\ogamma, I, J)} \circ \twist}"]
&
&
&
\overline{h}[\ogamma]
\end{tikzcd}
\end{equation}

We note that absorbing the twist in the bottom left vector space and erasing the middle line gives us the desired diagram.
\end{proof}

\begin{prop}\label{prop:multicharHopf}
Consider a combinatorial Hopf monoid $(\overline{h}, \eta ) $.
Let $\opi, \otau $ be set compositions of the disjoint non-empty sets $I$ and $J$, respectively, and take $\olambda $ set composition of S = $I \sqcup J$.
Take $a\in \overline{h}[I]$, $b\in \overline{h}[J]$, $c\in \overline{h}[I\sqcup J]$.
Then we have that 
\begin{equation}\label{eq:eq1char}
f_{ \olambda, \eta} (a \cdot_{I, J} b) = f_{\olambda|_I, \eta}(a) f_{ \olambda|_J, \eta}(b)\, ,
\end{equation}
and that
\begin{equation}\label{eq:eq2char}f_{\opi, \eta } \otimes f_{\otau, \eta } \circ \Delta_{I, J } ( a) = f_{(\opi , \otau ), \eta} ( a ) \, .
\end{equation}
\end{prop}

\begin{proof}
Note that \eqref{eq:eq1char} reduces to 
\begin{equation}\label{eq:redproofcharmorph} f_{\ogamma, \eta} \circ \mu_{I, J} = f_{\ogamma|_I, \eta} \otimes f_{\ogamma|_J, \eta}.
\end{equation}

Now \cref{prop:4diaggen} tells us that 
\begin{equation}\label{eq:1proof1}\Delta_{\ogamma } \circ \mu_{I, J} = \mu_{(\ogamma , I, J)} \circ (\Delta_{\ogamma|_I} \otimes \Delta_{\ogamma|_J})\, .
\end{equation}

Suppose that $\ogamma = C_1 | \dots | C_k$ for $k\geq 1$.
Then by tensoring diagrams of the form \eqref{cd:chars} for each decomposition $(I \cup C_i)\sqcup (J \cup C_i) = C_i$, we obtain 
\begin{equation}\label{eq:1proof2} \eta[\ogamma ] \circ \mu_{(\ogamma, I, J)} = \eta[(\ogamma|_I,  \ogamma|_J)] \, .
\end{equation}

From \eqref{eq:1proof1} and \eqref{eq:1proof2} we get that 
$$\eta[\ogamma ]\circ \Delta_{\ogamma }\circ \mu_{I, J} = \eta[\ogamma ] \circ \mu_{(\ogamma , I, J)} \circ (\Delta_{\ogamma|_I} \otimes \Delta_{\ogamma|_J}) =\eta[(\ogamma|_I,  \ogamma|_J)]  \circ (\Delta_{\ogamma|_I} \otimes \Delta_{\ogamma|_J}) \, .$$

So
$$f_{\ogamma, \eta } \circ \mu_{I, J} = (\eta[\ogamma|_I] \otimes \eta[ \ogamma|_J] )\circ (\Delta_{\ogamma|_I} \otimes \Delta_{\ogamma|_J}) = \eta_{\ogamma|_I} \otimes \eta_{\ogamma|_J}\, . $$

This concludes the proof of \eqref{eq:redproofcharmorph}.
Remains to show \eqref{eq:eq2char}, which follows from \eqref{eq:multicoass} via

$$( f_{\opi, \eta} \otimes f_{\otau, \eta} )\circ \Delta_{I, J} = ( \eta[\opi] \otimes \eta[\otau] )\circ ( \Delta_{\opi} \otimes \Delta_{\otau} ) \circ \Delta_{I, J} =   \eta[(\opi , \otau)] \circ \Delta_{(\opi,  \otau ) } = f_{(\opi  , \otau ), \eta } \, ,$$
whenever both $I$ and $J$ are non empty, establishing the equality as desired.
\end{proof}

\begin{proof}[Proof of \cref{thm:universalhs}]
Let $a \in \overline{h}[I]$.
We define
$$\uhsm_{\overline{h}}( a) = \sum_{\opi \in \oPi_I} \mathbb{M}_{\opi}\eta_{\opi}(a) \, ,$$

The commutativity of Diagram \eqref{cd:uhsm} follows because $f_{\opi, \eta }= \eta $ whenever $\opi $ has length $1$ or $0$.
Remains to show that such map is a combinatorial Hopf monoid morphism, i.e. that we have: 
\begin{itemize}
    \item $\uhsm_{\overline{h}} ( a \cdot_{I, J} b ) = \uhsm_{\overline{h}} (a)\cdot  \uhsm_{\overline{h}} (b) $.
    
    \item $\uhsm_{\overline{h}\cdot \overline{h}} (\Delta_{I, J} a ) = \Delta_{I, J} \circ\uhsm_{\overline{h}} (a) .$
    
    \item $\uhsm_{\overline{h}} \circ \iota_{\overline{h}}  = \iota_{\hs{WQSym}} $.
    
    \item $ \epsilon_{\hs{WQSym}} \circ \uhsm_{\overline{h}} = \epsilon_{\overline{h}}$.
\end{itemize}

The last two equations follow from direct computation.
Taking the coefficients on the monomial basis for the first two items, this reduces to \cref{prop:multicharHopf} whenever $I, J$ are non empty.

Further, when wlog $I=\emptyset $, we can assume $a= \lambda \mathbb{M}_{\vec{\emptyset }} $ and the first equation follows immediately.
Also, when $I=\emptyset $, $\Delta_{I, J} (a) = a \otimes \mathbb{M}_{\vec{\emptyset }} $, and the second equation follows.
This concludes that $\uhsm_{\overline{h}} $ is a combinatorial Hopf monoid morphism.

It remains to establish the uniqueness.
Suppose that $\phi: \overline{h} \Rightarrow \hs{WQSym}$ is a combinatorial Hopf monoid morphism.
Note that $\eta_0 [\emptyset]$ is an isomorphism, so from \eqref{cd:uhsm} applied to both $\phi $ and $\uhsm_{\hs{h}}$ we get $\eta_0 [\emptyset]^{-1} \eta[\emptyset] = \phi[\emptyset ] = \uhsm_{\hs{h}}[\emptyset ]$.

Now take $I$ non-empty.
For each $a \in \overline{h}[I]$, write $\phi[I] ( a ) = \sum_{\opi \in \oPi_I } \mathbb{M}_{\opi } \phi_{\opi } (a) $ and apply $\Delta_{\opi} $ on both sides.
Since $\phi$ is a comonoid morphism, we have:

\begin{equation}\label{eq:basephi}
\phi[\opi] \Delta_{\opi } ( a ) = \Delta_{\opi } \phi[I] ( a ) = \sum_{\otau } \phi_{\otau}(a) \Delta_{\opi } \mathbb{M}_{\otau}  = \phi_{\opi}(a) \Delta_{\opi } \mathbb{M}_{\opi} \, ,
\end{equation}
because $\Delta_{\opi}(\mathbb{M}_{\otau}) = 0$ whenever $\opi \neq \otau $ and $|\opi | = |\otau |$.

However, since $\eta_0 \phi = \eta $, we have that $\eta_0[\opi ] \phi[\opi ] = \eta[\opi ]$, so 

\begin{equation}\label{eq:canophi}
 \eta_0[\opi] \phi[\opi ] \Delta_{\opi } (a) = \eta[\opi ] \Delta_{\opi } (a)  = f_{\opi, \eta } (a) \, .
\end{equation}

Applying $\eta_0[\opi ] $ on \eqref{eq:basephi} and using \eqref{eq:canophi}, gives us $ \phi_{ \opi } ( a) f_{\opi, \eta_0 }(\mathbb{M}_{\opi}) =\eta_{\opi } a  $.
But $f_{ \opi, \eta_0 }(\mathbb{M}_{\opi}) = 1$, so we have that 
$$\phi_{ \opi } ( a)  =f_{\opi, \eta }(a) \, .$$
which concludes the uniqueness. 
\end{proof}

\begin{rem}\label{rem:aguiarwasright}
We would like to point out that this theorem also follows from \cite[Theorem 11.23]{aguiar10} directly.
However, we present here a proof with multicharacters in the interest of self containment.

Specifically, let $(\eta, h)$ be a combinatorial Hopf monoid, and consider the following Hopf submonoid $h'$, where $h'[I]=h[I]$ for $I \neq \emptyset$, and $h'[\emptyset ] = 1\mathbb{K}$.
This is also a combinatorial Hopf monoid, as the maps $\epsilon $ and $\eta $, suitably redefined to be the relevant restriction to $h'[\emptyset ]$, satisfy the Hopf monoid axioms.
The unit and bialgebra axioms guarantee that $h'[\emptyset ] $ is stable for the product and the coproduct.

Therefore, because this a connected Hopf monoid, from \cite[Theorem 11.23]{aguiar10} there is a unique Hopf monoid morphism $\Psi: h' \Rightarrow \hs{WQSym}$.
This map can be extended to a Hopf monoid morphism $\Psi:h \Rightarrow \hs{WQSym}$ by setting $\Psi_{\emptyset} = \eta_{\emptyset}$.
It must be checked that this map satisfies the Hopf monoid morphism axioms.
It is a direct observation that \eqref{cd:uhsm} commutes, making this a Hopf monoid morphism.

Conversely, if there are two distinct combinatorial Hopf monoids $\Psi_1, \Psi_2 : h \Rightarrow \hs{WQSym}$, then they must agree on $h[\emptyset ]$ in accordance with \eqref{cd:uhsm}.
On the other hand, if we consider the compositions $\inc \circ \Psi_1$, $ \inc \circ \Psi_2$, these correspond to two combinatorial Hopf monoid morphisms $h ' \Rightarrow \hs{WQSym}$, so they must coincide in accordance with \cite[Theorem 11.23]{aguiar10}.
Thus, $\Psi_1, \Psi_2$ must agree on $h[I]$ for any non-empty set $I$.
\end{rem}

\subsection{Generalized permutahedra and posets}\label{sec:gperandPosets}
In the following, we see that the universal map that we constructed above in \cref{thm:universalhs} is well behaved with respect to combinatorial Hopf monoid morphisms.
This is in fact a classical property of terminal objects in any category.

\begin{lm}\label{lm:liftchars}
If $\phi : \hs{h}_1 \Rightarrow \hs{h}_2 $ is a combinatorial Hopf monoid morphism between two Hopf monoids with characters $\eta_1$ and $\eta_2 $ respectively, then the following diagram commutes

\begin{equation}\label{cd:chrommorph}
\begin{tikzcd}
\hs{h}_1
\arrow[rr, "{ \phi}"] 
\arrow[rd, "{ \mathbf{\Psi }_{\hs{h}_1}}"] 
& 
&
\hs{h}_2
\arrow[dl, "{ \mathbf{\Psi }_{\hs{h}_2} }"] \\
&
\hs{WQSym}
&
\end{tikzcd}
\end{equation}
\end{lm} 

\begin{proof}
It is a direct observation that the composition of combinatorial Hopf monoids morphism is still a combinatorial Hopf monoid morphism.
Hence, we have that $\uhsm_{\hs{h}_2} \circ \phi : \hs{h}_1 \to \hs{WQSym}$ is a combinatorial Hopf monoid morphism.
By uniqueness it is $\uhsm_{\hs{h}_2} \circ \phi= \uhsm_{\hs{h}_1}$.
\end{proof}

\begin{cor}\label{cor:PosEmbedding}
There are no Hopf monoid morphisms $\phi : \hs{Pos} \to \hs{GP}$ that preserve the corresponding characters.
\end{cor}

\begin{proof}
For sake of contradiction, suppose that such $\phi$ exists, hence it satisfies $\fff(\uhsm_{\hs{GP}} ) \circ \fff(\phi ) = \fff (\uhsm_{\hs{Pos}}) $, according to \cref{lm:liftchars}, so

$$\uhsm_{\mathbf{Pos}} = \uhsm_{\mathbf{GP } } \circ  \fff(\phi ) \,. $$

However, $\uhsm_{\mathbf{Pos}} $ is surjective, whereas we have seen in \cref{thm:imGP} that $\uhsm_{\mathbf{GP } } $ is not surjective.
This is the desired contradiction.
\end{proof}

\section*{Acknowledgement}

The author would like to thank the support of the SNF grant number 172515.
This work was brought up in the encouraging environment provided by V. F\'eray, who gave invaluable suggestions.

The author would also like to thank the referees, for their helpful comments that improved readability, clarity and mathematical rigor.
Particularly, to the referees for pointing out the simple proof using \cite[Theorem 11.23]{aguiar10}.

\appendix

\section{Computing the augmented chromatic symmetric function on graphs\label{app:coef}}

Recall that in \cref{sec:3}, we define the ring $R$ as the quotient ring of power series in $\mathbb{K}[[x_1, \dots ; q_1, \dots ]]$ by the relations $q_i(q_i-1)^2 = 0 $.
We are then able to define a map $\tilde{\Psi} : \mathbf{G} \to R $, and we observed that $\ker \tilde {\Psi } = \ker \Psi$ in \cref{prop:goodinv}.

Here, we consider some specializations of $\tilde{\Psi } $  and obtain a linear combination of chromatic symmetric function of smaller graphs, in \cref{thm:tildeinpsi}.
The main motivation is to explore how to obtain \cref{prop:goodinv} without using \cref{thm:graphkernel}, and instead use a more direct way.
This is not established in this paper, which illustrates the strength of the kernel approach.
Let us first set up some necessary notation.

For an element $f \in R$, denote by $f\big|_{q_i = a} $ the specialization of the variable $q_i $ to $a$ in $f$, whenever defined (for $a = 0$ or $a=1$).
Additionally, denote by $f\big|_{q_i = 1'} $ the specialization of the variable $q_i $ to $1$ in $\frac{\partial}{\partial q_i} f$.
This is naturally an abuse of notation that allows us to denote the composition of several specializations in a more compact way.
We also use this notation for the $x_i $ variables.
Further, we denote by $f\big|_{x_i = 0''}$ the specialization $\frac{\partial^2}{\partial x_i^2} f\big|_{x_i=0}$.

We note that, in this ring, we can specialize infinitely many variables to zero.
This however cannot be done with specializations to one, as the reader can readily check.
Taking specializations of $q_i $ to $a \not \in \{ 0, 1, 1' \} $ is not well defined in the quotient ring.

For an edge $e \in E(G)$, denote by $G\setminus \mathcal{N}(e) $ the graph resulting after both endpoints of $e$ are deleted from $G$, along with all its incident edges.

We say that a tuple of edges $m =(e_1, \dots , e_k) $ is an ordered matching if no two edges share a vertex, and write $\mathcal{M}_k(G)$ for the set of ordered matchings of size $k$ on a graph $G$.
We write $G\setminus \mathcal{N} ( m )$ for the graph resulting after removing all vertices in the matching $m$ from $G$, along with all its incident edges.

Finally, for a symmetric function $f$ over the variables $x_1, x_2, \dots $, let $f\uparrow_k $ be the symmetric function over the variables $x_{k+1}, x_{k+2}, \dots $ with each index in $f$ shifted up by $k$.

We obtain now a formula for $\tilde{\Psi}(G)$ that depends only on $\Psi_{\gHa} (H_i) $ for some graphs $H_i$ that have less vertices than $G$.

\begin{thm}\label{thm:tildeinpsi}
Let $k \geq 0$.
We have the following relation between the graph invariant $\tilde{\Psi }$ and the chromatic symmetric function $\Psi_{\gHa }$:
\begin{equation}\label{eq:specializing}
\frac{1}{2^k} \tilde{\Psi}(G)\Big|_{\substack{q_i = 1' \, \, i =1, \dots , k \\ q_i=0 \,\, i>k \\x_i= 0'' \, \, i =1, \dots , k}} =  \sum_{m \in \mathcal{M}_k(G)} \Psi_{\gHa}(G\setminus \mathcal{N}(m))\uparrow_k  \, .
\end{equation}
\end{thm}

\begin{proof}
Recall that $c_G(f, i)$ counts the number of monochromatic edges in $G$ with color $i$.
With the expression given in \cref{sec:3} for the augmented chromatic symmetric function, we have
$$ \tilde{\Psi} (G) \Big|_{q_i = 1' \, \, i =1, \dots , k }  = \sum_{f:V(G) \to \mathbb{N} } x_f \left(\prod_{i > k} q_i^{c_G(f, i) } \right) \prod_{i=1}^k c_G(f, i) \, .$$

We say that a coloring of $G$ is $k$\textit{-proper} if all monochromatic edges have color $j \leq k$.
Observe that for a fixed coloring $f$, $\prod_{i=1}^k c_G(i, f) $ counts ordered matchings $(e_1, \ldots , e_k )$ in $G$ that satisfy $f(v) = \{ i \}$ for any vertex $v$ of $e_i$, $i=1, \dots , k$.
Then, it is clear that 
\begin{equation}
\begin{split}
\tilde{\Psi} (G) \Big|_{\substack{q_i = 1' \, \, i =1, \dots , k\\  q_i=0 \, \, i>k }} &= \sum_{f \text{ is k - proper } } x_f \prod_{i=1}^k c_G(i, f) = \sum_{f \text{ is k - proper }} \left[ x_f \, \, \, \sum_{\substack{m \in \mathcal{M}_k(G) \\ m = (e_1, \dots , e_k)  \\ f(e_i) = \{ i \} }} 1\right] \\
  &= \sum_{\substack{m \in \mathcal{M}_k(G) \\ m = (e_1, \dots , e_k) }} \left[ \, \, \, \sum_{\substack{f \text{ is k - proper }\\ f(e_i) = \{ i \} } } x_f  \right] \\
  &= \sum_{\substack{m \in \mathcal{M}_k(G) \\ m = (e_1, \dots , e_k) } } \, \,  \, \sum_{\substack{g \text{ coloring in } G\setminus \mathcal{N}(m) \\ g \text{ is k - proper}  } } x_g (x_1\cdots x_k)^2 \, .
\end{split}
\end{equation}

So after the specialization $x_i = 0'' $ for $i=1, \dots , k$, all colorings of $G\setminus \mathcal{N}(m) $ that use a color $j \leq k $ vanish, so
\begin{equation}
\begin{split}
\tilde{\Psi} (G) \Big|_{\substack{q_i = 1' \, \, i =1, \dots , k \\ q_i=0 \,\, i>k \\x_i= 0'' \, \, i =1, \dots , k}} =& \sum_{\substack{m \in \mathcal{M}_k(G) \\ m = (e_1, \dots , e_k) } } \, \, \, \sum_{\substack{g \text{ proper in } G\setminus\mathcal{N}(m)  \\ \im g \subseteq \mathbb{Z}_{>k}  } }
2^k x_g \\
				=& 2^k \sum_{m \in \mathcal{M}_k(G) }  \Psi_{\gHa}( G\setminus \mathcal{N}(m)  )\uparrow_k 
\end{split}
\end{equation}
as desired.
\end{proof}

The right hand side of the expression of \cref{thm:tildeinpsi} can, in fact, be determined by the chromatic symmetric function of the graph $G$.

\begin{prop}\label{prop:RHSthm}
If $G_1, G_2$ are two graphs such that $\Psi_{\gHa}(G_1 ) = \Psi_{\gHa} (G_2) $, and $k$ a positive number, then 
$$ \sum_{m \in \mathcal{M}_k(G_1)} \Psi_{\gHa}(G_1\setminus \mathcal{N}(m)) = \sum_{m \in \mathcal{M}_k(G_2)} \Psi_{\gHa}(G_2\setminus \mathcal{N}(m)) \, . $$
\end{prop}

\begin{proof}
We use the power-sum basis $\{p_{\lambda } \}_{\lambda \vdash n} $ of $Sym_n$ introduced in \cite{stanley00}.

We show that for a generic graph $H$, the coefficients of the symmetric function $ \sum_{m \in \mathcal{M}_k(H)} \Psi_{\gHa}(H\setminus \mathcal{N}(m)) $ in the power-sum basis are a function of the coefficients of $\Psi_{\gHa } (H)$ in the power-sum basis.
Once this is established, the proposition follows.

For a graph $G$ and a set of edges $S \subseteq E(S)$, write $\tau (S) $ for the integer partition recording the size of the connected components of the graph $(V(G), S)$.
In \cite{stanley95} the following expression for the coefficients in the power-sum basis is shown:
$$\Psi_{\gHa} (G) = \sum_{\lambda \vdash n} p_{\lambda } \sum_{\substack {S \subseteq E(G) \\ \tau (S) = \lambda  }} (- 1)^{\# S } \, .$$

Suppose that $\Psi_{\gHa }(H) = \sum_{\lambda \vdash n} c_{\lambda } p_{\lambda }$.
For an integer partition $\lambda $ write $m_2(\lambda ) $ for the number of parts of size two, and write $\lambda \cup (2^k) $ for the integer partition resulting from $\lambda $ by adding $k$ extra parts of size two.
Then we have that 
\begin{equation}
\begin{split}
\sum_{m \in \mathcal{M}_k(H)} \Psi_{\gHa}(H\setminus  \mathcal{N}(m)) &= \sum_{m \in \mathcal{M}_k(H)} \sum_{\lambda \vdash n} p_{\lambda } \sum_{\substack{ S \subseteq E( H\setminus  \mathcal{N}(m))  \\ \tau (S) = \lambda }} (-1)^{\# S}  \\
																&= \sum_{\lambda \vdash n} p_{\lambda } \sum_{m \in \mathcal{M}_k(H)} \sum_{\substack{ S \subseteq E( H\setminus  \mathcal{N}(m) )  \\ \tau (S) = \lambda }} (-1)^{\# S}\, .
\end{split}
\end{equation}

Relabel the summands index by setting $R = S \cup m $, and note that for each $R\subseteq E(H) $ such that $\tau (R) = \lambda \cup (2^k )$, there are exactly $  \binom{m_2 (\lambda ) +  k}{ k} k! $ pairs $(S, m) $ of $m\in \mathcal{M}_k(H) $ and $S \subseteq E(H) $ such that  $S \cup m = R $, $\tau (S) = \lambda $ and $S \subseteq E(H \setminus  \mathcal{N}(m) )$.
Hence:
 
\begin{equation}
\begin{split}													
\sum_{m \in \mathcal{M}_k(H)} \Psi_{\gHa}(H\setminus  \mathcal{N}(m)) &= \sum_{\lambda \vdash n} p_{\lambda } \sum_{\substack{R \subseteq E(H) \\ \tau ( R) = \lambda \cup 2^k }} \binom{m_2 (\lambda ) +  k}{ k} k! (-1)^{\# R - k} \\
																&= \sum_{\lambda \vdash n} p_{\lambda } (-1)^k c_{\lambda \cup ( 2^k )} \binom{m_2 (\lambda ) + k }{ k } k! \, \, .
\end{split}
\end{equation}

Therefore, the sum $\sum_{m \in \mathcal{M}_k(H)} \Psi_{\gHa}(H\setminus  \mathcal{N}(m))$ is determined by $\Psi_{\gHa}(H)$.
\end{proof}

It follows from \cref{thm:tildeinpsi} and \cref{prop:RHSthm} that:

\begin{cor}\label{cor:weakgoodiv}
If $G_1, G_2$ are graphs such that $\Psi_{\gHa}(G_1) = \Psi_{\gHa}(G_2)$, then for every integer $k \geq 0 $ we have:
$$ \tilde{\Psi}(G_1)\Big|_{\substack{q_i = 1' \, \, i =1, \dots , k \\ q_i=0 \,\, i>k \\x_i= 0'' \, \, i =1, \dots , k}} = \tilde{\Psi}(G_2)\Big|_{\substack{q_i = 1' \, \, i =1, \dots , k \\ q_i=0 \,\, i>k \\x_i= 0'' \, \, i =1, \dots , k}} \, .$$
\end{cor}

The following fact is immediate from the definition of $R$:

\begin{prop}\label{prop:rinlimmit}
Suppose that $f_1, f_2 \in R $ are such that for every pair of finite disjoint sets $I, J \subseteq \mathbb{N} $ we have
$$f_1 \Big|_{\substack{q_i = 1' \, \, i \in I \\ q_i=0 \,\, i \in P \\q_i=1 \, \, i \in J}} = f_2 \Big|_{\substack{q_i = 1' \, \, i \in I \\ q_i=0 \,\, i \in P \\q_i= 1 \, \, i \in J}} $$ 
where $P = \mathbb{N} \setminus (I \uplus J) $.
Then $f_1 = f_2 $ in $R$.
\end{prop}

In conclusion, to get an alternative proof of \cref{prop:goodinv} we need to establish a generalization of \cref{cor:weakgoodiv} that introduces specializations of the type $q_i = 1$, in order to apply \cref{prop:rinlimmit}.
Such a generalization has not been found by the author.

\bibliographystyle{alpha}
\bibliography{biblio}

\end{document}